\documentclass[a4paper,12pt]{amsart}

\usepackage[a4paper]{geometry}
\usepackage[utf8]{inputenc}
\usepackage[OT2,T1]{fontenc}

\usepackage{lipsum}\usepackage[english]{babel}

\usepackage{amssymb,amsmath,amsthm,amscd}
\usepackage{mathrsfs}

\usepackage{tikz-qtree}
\usepackage{amsmath,amsfonts}
\usetikzlibrary{shapes,arrows.meta}
\usepackage{subfig}
\usepackage{tabularx}
\usepackage{calc} 
\usepackage{graphicx}
\usepackage{hyperref}

\usepackage{longtable}
\usepackage{enumerate}

\usepackage{array}
\usepackage{mathrsfs}
\usepackage{version}
\usepackage{lscape}
\usepackage{epigraph}
\usepackage[all]{xy}
\usepackage{datetime}
\usepackage{todonotes}
\usepackage{graphicx}
\usepackage{multirow}
\usepackage{tikz}
\usepackage{pgf,tikz}
\usepackage{titlesec}
\usepackage{titletoc}
\titleformat{\section}{\centering\large}{\thesection.\ }{0em}{\MakeUppercase}[\titlerule]
\titleformat{\subsection}{\bf\centering\normalsize}{\thesubsection.\ }{0em}{}
\titleformat{\subsubsection}{\bf\centering\small}{}{0em}{}
\usetikzlibrary{positioning}
\tikzset{
	base/.style={draw, 
		align=center, minimum height=2ex},
	box/.style={base, rectangle, text width=2.5em, fill=red!50,font={\scriptsize}},
	upline/.style={line width=1.5pt}
}
\newcommand{\orbitthree}[9]{%
	\def\xstep{2}
	\def\x0{0}
	\def\y0{0}
	\coordinate (a1) at (\x0+#8*\xstep,\y0+#9*\xstep);
	\coordinate (a2) at (\x0+\xstep+#8*\xstep,\y0+#9*\xstep);
	\coordinate (a3) at (\x0+2*\xstep+#8*\xstep,\y0+#9*\xstep);
	\node[above left = 1em  and 3em of a1] (10) {$1_0$};
	\node[below left = 1em  and 3em of a1] (20) {$2_0$};
	\node[above right = 1em  and 3em of a3] (1inf) {$#1_\infty$};
	\node[below right = 1em  and 3em of a3] (2inf) {$#2_\infty$};
	
	\draw [upline](10)--(a1)node [pos=0.5, above =0.1] {$#3$};
	\draw (20)--(a1)node [pos=0.5, below=0.1] {$#4$};
	
	\draw (a1) .. controls (\x0+0.5*\xstep+#8*\xstep,\y0+0.25*\xstep+#9*\xstep)  .. (a2)
	(a2) .. controls (\x0+1.5*\xstep+#8*\xstep,\y0-0.25*\xstep+#9*\xstep)  .. (a3)
	(a3)--(1inf);
	
	\draw [upline]
	(a1) .. controls (\x0+0.5*\xstep+#8*\xstep,\y0-0.25*\xstep+#9*\xstep)  .. (a2)
	(a2) .. controls (\x0+1.5*\xstep+#8*\xstep,\y0+0.25*\xstep+#9*\xstep)  .. (a3)
	(a3)--(2inf);
	
	\draw node at (a1)[below=0.05]{$#5$};
	\draw node at (a2)[below=0.05]{$#6$};
	\draw node at (a3)[below=0.05]{$#7$};
}

\newtheorem{theorem}{Theorem}[section]
\newtheorem{lemma}[theorem]{Lemma}
\newtheorem{proposition}[theorem]{Proposition}
\newtheorem{corollary}[theorem]{Corollary}

\newtheorem{example}[theorem]{Example}

\newtheorem{remark}[theorem]{Remark}

\def\Z{\mathbb{Z}}

\def\Q{\mathbb{Q}}

\def\P{\mathbb{P}}
\def\E{\mathscr{E}}
\DeclareMathOperator{\Gal}{Gal}
\DeclareMathOperator{\myspan}{span}

\DeclareMathOperator{\diag}{diag}
\newcommand{\matr}[4]{\left(\begin{array}{cc}%
		#1 & #2 \\%
		#3 & #4 \\ %
	\end{array}\right)}

\def\RLL[#1]{R(\mathcal{L}(O_#1))}
\def\PO[#1]{\mathcal{P}(O_#1)}

\title{Geometry of the del Pezzo surface $y^2=x^3+Am^6+Bn^6$}

\author{Julie Desjardins}
\address{Mathematical and Computational Sciences\\ 
	University of Toronto Mississauga\\
	Deerfield Hall\\
	Mississauga, ON L5L 3E2\\
	Canada}
\email{julie.desjardins@utoronto.ca}

\author{Bartosz Naskręcki}
\address{Faculty of Mathematics and Computer Science\\
	Adam Mickiewicz University in Poznań\\
	ul. Uniwersytetu Poznańskiego 4 \\
	61-614, Poznań\\ 
	Poland}
\address{Mathematical Institute\\
	Polish Academy of Sciences\\ 
	ul. Śniadeckich 8\\
	00-656, Warszawa\\
	Poland}
\email[B. Naskręcki]{bartosz.naskrecki@amu.edu.pl}


\contentsmargin{2em} 
\dottedcontents{section}[3.0em]{}{2em}{0.5pc} \dottedcontents{subsection}[6em]{}{2em}{0.5pc}

\begin{document}
	
\tikzset{
	level distance=2cm,
	level 1/.style={sibling distance=1em,level distance =2cm},
	blank/.style={draw=none},
	edge from parent/.style=
	{->,-Latex, draw,edge from parent path={(\tikzparentnode) -- (\tikzchildnode)}}
}

\tikzset{
	base/.style={draw, 
		align=center, minimum height=2ex},
	proc/.style={base, rectangle, text width=2.5em, fill=red!50,font={\scriptsize}},
	test/.style={base, diamond, aspect=2, text width=7em, fill=blue!30,font={\scriptsize}},
	test2/.style={base, diamond, aspect=3.5, text width=6.5em,text height=0.5em, fill=blue!30,font={\tiny}},
	start/.style={proc, rounded corners,, fill=green!30},
	coord/.style={coordinate, on chain, on grid, node distance=6mm and 25mm},
	nmark/.style={draw, cyan, circle, font={\sffamily\bfseries}},
	norm/.style={->, draw, lcnorm},
	free/.style={->, draw, lcfree},
	cong/.style={->, draw, lccong},
	it/.style={font={\small\itshape}}
}

\begin{abstract}In this paper, we give an effective and efficient algorithm which on input takes non-zero integers $A$ and $B$ and on output produces the generators of the 
	Mordell-Weil group of the elliptic curve over $\Q(t)$ given by an equation of the form $y^2=x^3+At^6+B$. 
	Our method uses the correspondence between the 240 lines of a del Pezzo surface of degree 1 and the sections of minimal canonical height on the corresponding elliptic surface over $\overline{\Q}$.

For most rational elliptic surfaces, the density of the rational points is proven by various authors, but the results are partial in case when the surface has a minimal model that is a del Pezzo surface of degree 1. 
In particular, the ones given by the Weierstrass equation $y^2=x^3+At^6+B$, are among the few for which the question is unsolved, because the root number of the fibres can be constant. Our result proves the density of the rational points in many of these cases where it was previously unknown.
\end{abstract}
\maketitle

\setcounter{tocdepth}{2}
\tableofcontents

\section{Introduction}
\subsubsection*{Del Pezzo surfaces and rational elliptic surfaces}
A del Pezzo surface over a field $k$ is a smooth, projective, geometrically integral surface $X$ over $k$ with an ample anticanonical divisor $-K_X$. Del Pezzo surfaces are classified by their degree $d:=K_X^2$, an integer $1\leq d\leq9$. It has been proven that if $X$ is not geometrically $\overline{k}$-isomorphic to $\P^1\times\P^1$, then $X$ is $\overline{k}$-isomorphic to the blowup of $\P^2$ in $9-d$ points in ''general position'', cf. \cite[Chap. 8.1]{Dolgachev_geometry}.

An elliptic surface over $k$ with base $\P^1$ is a smooth projective surface $\E$ together with a map $\pi:\E\rightarrow\P^1$ such that the general fibre of $\pi$ is a smooth connected curve $E$ of genus one. 
We assume that $\pi$ has a section: $\E$ is thus a family of elliptic curves away from finitely many fibres, and admits a Weierstrass equation as an elliptic curve over $k(t)$. We say that $\E$ is rational if it is birational to $\P^2$. 

A rational elliptic surface is $\overline{k}$-isomorphic to $\P^2$ blown-up in the fundamental locus of a pencil of plane cubics: thus by choosing carefully $d$ points on a del Pezzo surface, one obtains a rational elliptic surface $\E$. 

In particular if $d=1$, then the blowup of the base point of the anticanonical linear system of $X$ allows to obtain a rational elliptic surface $\E$. In terms of equation: $X$ is isomorphic to a smooth sextic hypersurface in the weighted projective space $\P^1(1,1,2,3)$ defined by the equation $y^2=x^3+F(m,n)x+G(m,n)$, with $F,G$ homogeneous polynomials of degree respectively 4 and 6 (\cite[Theorem III.3.5]{Kollar}). The converse is also true.
The base point of the anticanonical system is $P:=[0,0,1,1]$: and the blow-up of $X$ at $P$ gives the surface $\E$ in $\P^3$ with equation $y^2=x^3+f(t)x+g(t)$ where $f(t)=F(t,1)$ and $g(t)=G(t,1)$. Then $\E$ is a rational elliptic surface (\cite{Miranda}).

A rational elliptic surface obtained through this process has only irreducible fibres (of Kodaira type $I_1$ or $II$). Moreover, it reaches the maximum of 8 independent rational sections over $\overline{k}$ of the Mordell-Weil lattice and the torsion subgroup of the generic fibre is trivial, cf. \cite{Oguiso_Shioda}.

In this paper, we are interested in certain rational elliptic surfaces $\E_{G}$ with a generic fibre of the form $$E_G:y^2=x^3+G(t),$$ where $G(t)\in\Z[t]$ is such that $1\leq\deg G(t)\leq 6$.
These elliptic surfaces have many particularities:
\begin{enumerate}
	\item if $G(t)$ is a squarefree polynomial of degree $\geq5$, then the contraction of the image of the zero section gives a del Pezzo surface of degree 1, cf. \cite[\S 8.8.3]{Dolgachev_geometry};
	\item $\E=\E_G$ is isotrivial, i.e. its fibres $\E_t$ are isomorphic to one another. Indeed, the $j$-invariant function $t\rightarrow j(\E_t)$ equals $0$. 
\end{enumerate}

\subsubsection*{Main theorem}

\begin{theorem}\label{thmdescription}
	Let $A,B\in\Z$ be non-zero rational integers and define $\E=\E_{A,B}$ to be the elliptic surface given by the equation \begin{equation}\label{eq:our_surfaces}
		E_{A,B}:y^2=x^3+At^6+B.
	\end{equation} 
	The rank of the generic fibre $E_{A,B}$ over $\overline{\mathbb{Q}}(t)$ is $8$ and the rank $r_{\mathcal{E}}$ of the group $E_{A,B}(\mathbb{Q}(t))$ is at most equal to $3$. There exists an effective and efficient algorithm which on input takes two non-zero integers $A$ and $B$ and produces on output the generators of the group $E_{A,B}(\mathbb{Q}(t))$.
\end{theorem}
\begin{remark}
	The algorithm of Theorem \ref{thmdescription} is effective and efficient in the following sense. For two integers $A,B$ on the input, in order to compute the rank of the group $E_{A,B}(\mathbb{Q}(t))$, the algorithm only requires the answer to the following questions
	\begin{enumerate}[(1)]
		\item Is $A$ (resp. $B$) a square or $-3$ times a square?
		\item Is $A$ (resp. $B$, resp. $4AB$) a cube?
	\end{enumerate}
	For each task in the list there is an algorithm which given the input integer $N$ returns the answer in $O(log(N)^{1+o(1)})$ steps, \cite{Bernstein_roots}.
\end{remark}

We denote by $r_{\mathcal{E}}$ the rank of the group $E_{A,B}(\mathbb{Q}(t))$ and call it the \textit{generic rank} of $\E$ over $\mathbb{Q}$.

\begin{figure}[htb]
	\centering
	\begin{tikzpicture}[sibling distance=2.2em,level distance=1.7cm]
		\Tree [.\node[start]{START};
		[.\node[test2]{$[\mathbb{Q}(\sqrt{A},\zeta_3):\mathbb{Q}]=4$}; 
		\edge node[auto=right]{YES};
		[.\node[proc]{$r\leq 2$};
		]
		\edge node[auto=left]{NO}; 
		[.\node[test2]{$[\mathbb{Q}(\sqrt{B},\zeta_3):\mathbb{Q}]=4$}; 
		\edge node[auto=right]{YES};  
		\node[proc]{$r\leq 2$};
		\edge node[auto=left]{NO};
		[.\node[test2]{$4AB$ is a cube?};
		\edge node[auto=right]{YES};  
		[.\node(Acub)[test2]{$A$ is a cube?};
		\edge node[auto=right]{YES};
		\node[proc]{$r=3$};
		\edge node[auto=left]{NO};  
		[.\node(Bcub)[test2]{$B$ is a cube?};
		\edge node[auto=right]{YES};  
		\node[proc]{$r=3$};
		\edge node[auto=left]{NO};  
		\node[proc]{$r=2$};
		]
		]	
		\edge node[auto=left]{NO};  
		\node[proc]{$r\leq 2$};	
		]
		] 
		]
		]
	\end{tikzpicture}
	\caption{Simplified version of the rank decision algorithm. Each diamond box is a query with possible yes or no answer. Red boxes explain what is the value of the generic rank $r=r_{\mathcal{E}}$ after each step.}\label{fig:Diagram_simp}
\end{figure}
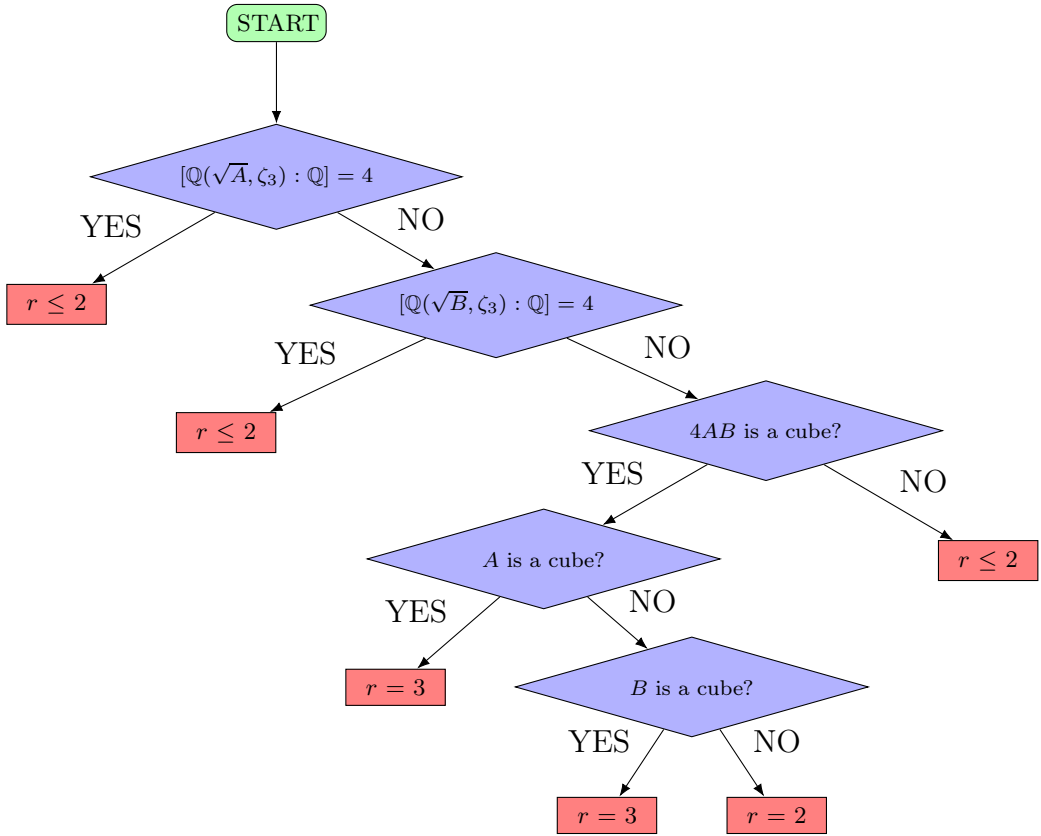

For instance, in order for the group $E_{A,B}(\mathbb{Q}(t))$ to have the maximum possible rank ($r_\E=3$), the coefficients $A$ and $B$ need both to be either a square or $-3$ times a square, $4AB$ needs to be a cube and either $A$ or $B$ needs to be a cube as well. That follows directly from the algorithm presented in Figure \ref{fig:Diagram_simp}. A detailed proof of the algorithm and the further steps which resolve questions about the generic rank $r\leq 2$ are described in Section \ref{sec:proof_main_theorem}.

\begin{theorem}
	Let $\E$ be the elliptic surface given by the equation $y^2=x^3+At^6+B$. Then
	$\E$ has generic rank $r_\E=3$ if there exists some $\alpha,\beta\in\Z$ such that one of the following holds:
	\begin{itemize}
		\item $A=\alpha^6$ and $B=2^4\beta^6$;
		\item $A=-3^3\alpha^6$ and $B=2^4\beta^6$;
		\item $A=-3^3\alpha^6$ and $B=-3^32^4\beta^6$;
		\item $A=\alpha^6$ and $B=-2^43^3\beta^6$;
		\item or one of the four previous cases, with roles of $A$ and $B$ interchanged.
	\end{itemize}
	Otherwise, $\E$ has rank $r_\E\leq2$.
\end{theorem}

\begin{remark}
	Theorem \ref{thmdescription} implies on an elliptic surface with generic rank $0$ that any minimal model of $\E$ is a del Pezzo surface of degree 1. This is due to the fact that Galois invariant part of the Picard group of $\E$ has rank $2$ (spanned by the image of the general fibre and the image of the zero section). The surface obtained by contraction of the zero section, has the Picard group over $\mathbb{Q}$ of rank $1$ and is minimal. This theory can be found in \cite{Manin}.\end{remark}

\begin{remark}
	In our main theorem, we choose $G(t)=C(AG_1(t)^2+BG_2(t)^2)$ with $G_1(t)=t^3$ and $G_2(t)=1$. An important remark to make is that our result will still hold for any $G_1(t)=L_1(t)^3$, $G_2(t)=L_2(t)^3$ for $L_1,L_2\in\Z[t]$ non proportional linear polynomials. The basis will be given by similar points, with a change of variable, cf. Corollary \ref{cor:lin_subs}. 
	However, given general polynomials $G_1(t),G_2(t)\in\Z[t]$, the situation differs. The Galois action on the Mordell-Weil lattice is maximal for a generic choice of $G_1$ and $G_2$. Points of the N\'{e}ron-Tate height $2$ form a big orbit of size $240$ in that case. We recall in Section \ref{sec:preliminaries_on_MW_groups} how to compute these heights \`{a} la Shioda.
	For particular choices of $G_1$ and $G_2$, it is possible to split this orbit further, for instance into three Galois orbits of size respectively $6$, $72$ and $162$ for most values of $A,B$. 
	These orbits can be decomposed even further by choosing appropriate values of $A$ and $B$ until finding a rational section. \end{remark}

A fractional linear change of coordinates $\phi:t\mapsto (at+b)/(ct+d)$ for $ad-bc\neq 0$ on the base $\mathbb{P}^{1}$ produces from the elliptic surface $\pi:\mathcal{E}_{G}\rightarrow\mathbb{P}^{1}$ a new elliptic surface $\pi':\mathcal{E}\rightarrow\mathbb{P}^{1}$ where $\pi'=\phi\circ \pi$. The map $\phi$ is an automorphism of $\mathbb{P}^{1}$ which induces the automorphism of elliptic surface $\mathcal{E}_{G}$ and $\mathcal{E}$. The effect of this change of coordinates on the generic fibre is visible by replacing the polynomial $G(t)$ with the polynomial $H(t)=(ct+d)^6 G\left(\frac{(at+b)}{(ct+d)}\right)$. Both elliptic surfaces have the same arithmetic properties which is visible in the following corollary.
\begin{corollary}\label{cor:lin_subs}
	Let $a,b,c,d\in\mathbb{Z}$ such that $ad-bc\neq 0$ and let $A,B\in\mathbb{Z}\setminus\{0\}$.
	Define $\E$ to be the elliptic surface given by the equation
	$$E:y^2=x^3+A(at+b)^6+B(ct+d)^6.$$
	The rank of the generic fibre $E$ over $\overline{\Q}(t)$ is 8 and the rank of the group $E(\Q(t))$ is at most equal to 3. The algorithm described in section \ref{sec:Decision_diagram} gives the rank of the group $E(\Q(t))$.
\end{corollary}

\subsubsection*{Motivation and method}
The motivation for Theorem \ref{thmdescription} is Corollary \ref{corollaryconstantrootnumber} that lists the elliptic surfaces $\E_{3a^2,b^2}$ ($a,b\in\mathbb{Z}\setminus\{0\}$) and their generic rank. 
As explained in Section \ref{section:previousapproaches}, on those rational elliptic surfaces the Zariski density of the rational points (and $\Q$-unirationality) is uncertain: there exists no geometric proof, and the study of the variation of the root number on the fibres is indecisive: it always takes the value $+1$, so according to the parity conjecture (weak BSD) the rank is even - but possibly zero on all the fibres.

Our method uses the correspondence between the 240 lines of a del Pezzo surface of degree 1 and the sections of minimal canonical height ($=2$) on the corresponding elliptic surface over $\overline{\mathbb{Q}}$. After finding a basis of the Mordell-Weil lattice over $\overline{\Q}$ we study the Galois action on those lines and the submodules their orbits generate in order to determine a basis (and so the rank) of the Mordell-Weil lattice over the rationals.

\subsubsection*{Contents of the paper}
First, in Section 2, we discuss some previous approaches towards proving the density of points on del Pezzo surfaces. Next, we study when the root number of the family $\mathcal{E}_{G}$ is positive. 
In those cases we can apply the algorithm from further sections to find instances when the Mordell-Weil rank over $\mathbb{Q}(t)$ is positive, hence obtaining a density of rational points on the surface $\mathcal{E}_{G}$. 
We discuss in Section \ref{sec:new_examples} for which polynomials $G$ we obtain a new density result.

In Section \ref{sec:structure_of_the_orbits} we compute the structure of the group $\Lambda=E_{G}(\overline{\mathbb{Q}}(t))$ for $G=At^6+B$, $A,B\in\mathbb{Q}^{\times}$.
We obtain it by studying the explicit set of points in $\Lambda$ which have height $2$. From the setup, there are $240$ such points and under the natural action of the absolute Galois group $\Gal(\overline{\mathbb{Q}}/\mathbb{Q})$ they decompose into at least $8$ orbits which we explicitly classify.

In Section~\ref{sec:proof_main_theorem},  based on the results of the previous section, we determine the structure of $\Lambda$ as the natural Galois module with respect to the group $\Gal(\overline{\mathbb{Q}}/\mathbb{Q})$. Since the group $\Lambda$ has no non-zero torsion elements it forms with a height pairing a positive definite lattice. We find a sublattice $\Lambda'$ of index $81$ in $\Lambda$ which allows us to compute the subgroup $\Lambda_{\mathbb{Q}}$ of points defined over $\mathbb{Q}(t)$ for any choice of $A,B\in\mathbb{Z}\setminus\{0\}$. 
The rank of $\Lambda_{\mathbb{Q}}$ varies between $0$ and $3$ and for each choice of $A$ and $B$ we provide a complete answer, packaged in a decision diagram in Section \ref{sec:Decision_diagram}. 
Next, we compute in Section \ref{sec:rational_basis} the minimal height generators for the group $\Lambda_{\mathbb{Q}}$. 

Finally, in Section \ref{sec:density_gen_rank_0} we briefly discuss what are the possible extensions of our work to the cases when the rank of $\Lambda_{\mathbb{Q}}$ is $0$.

In each section of the paper we verify certain statements with MAGMA computer algebra system. The source code of our programs is available online on the website of one of the authors, \cite{Magma_code}
\section{Previous and new approaches to density of points}\label{section:previousapproaches}

\subsection{Geometric methods}

Let $\E$ be a rational elliptic surface defined over $\Q$.
A famous theorem of Iskovskikh \cite{Isk} then says that $\E$ has a minimal model over $\Q$, denoted $X$, that is either a conic bundle of degree $\geq1$, or a Del Pezzo surface.

We say that a surface $S$ is $k$-unirational if there is a dominant rational map $\mathbb{P}^2\dashrightarrow S$. Be aware that for $k$ an algebraically closed field, being $k$-unirational is equivalent to being rational. However, it is not clear that it is the case for other fields. Moreover, $k$-unirationality is \textit{a priori} a stronger property, implying the density of the rational points.

\begin{theorem}[\cite{KollMell, Manin, STVA}]
	Let $\E$ be a rational elliptic surface over a field $k$ whose generic rank is non-zero, then the surface is $k$-unirational.
\end{theorem}

\begin{proof}
	Koll\'ar and Mella \cite{KollMell} proves $k$-unirationality when $X$ is a conic bundle\footnote{In this case, the surface is $\Q$-unirational, i.e. it is dominated by the projective plane $\P^2\dashrightarrow X$} when $\rm{char} k\not=2$ ,
	by Segre and Manin \cite{Manin} when $X$ is a del Pezzo surface of degree $d\geq3$ with at least one rational point, by Salgado, Testa and V\'arilly-Alvarado \cite{STVA}, based on a work of Manin \cite[Thm 29.4]{Manin}, when $X$ is a del Pezzo surface of degree 2 provided that it contains a rational point that neither lies on four exceptional curves nor the ramification quartic curve. Thus in that case, $k$-unirationality holds on both surfaces. 
	
	In the case of a del Pezzo surface of degree 2 that is obtained by contracting an exceptional curve of a del Pezzo surface of degree 1 to a point $p$, \cite[Corollary 14]{STVA} guaranties that $p$ is rational and does not lie in an exceptional curve or the ramification curve. 
\end{proof}

\begin{corollary}
	For the surfaces in our paper, this means Theorem \ref{thmdescription} proves $\Q$-unirationality on $\E_{A,B}$ provided that the rank is $\geq1$.
\end{corollary}

If the minimal model $X$ is a del Pezzo surface of degree $1$, the surface $X$ has automatically a rational point: the base point of the anticanonical system. However, the results concerning density of rational points are still partial.

In papers of Ulas \cite{Ulas,Ulas2} and Jabara \cite{Jabara}, the density of rational points on certain families of isotrivial rational elliptic surfaces with $j$-invariant $0$ and $1728$ is proved by constructing a multisection with infinitely many rational points on those families. 
An article of Salgado and van Luijk \cite{SVL} improves this construction, and proves the Zariski density of the set of rational points of a del Pezzo surface of degree 1 satisfying certain conditions. 
However, those conditions were hard to check and moreover, the multisection they constructed sometimes failed to have infinitely many rational points - and so it did not prove the Zariski density.

Recently, Bulthuis and van Luijk \cite{BulthuisVanLuijk} proved that given a point on a del Pezzo surface of degree 1 which is of finite order in the fibre, there exists an pencil of elliptic curves through this point: each of the elliptic curves of the fibration is a multisection of the surface. 
In order to prove the density of points one has to prove that one of these elliptic curves has positive Mordell-Weil rank over $\mathbb{Q}$. However, this last step is not always an easy task. 

In an article by Winter and the first author \cite{DW}, one proves the density of the rational points on elliptic surfaces of the form (\ref{eq:our_surfaces}) via the construction of an ''elliptic multisection'' passing through a non-torsion point, cf. Section \ref{sec:density_gen_rank_0}.

\subsection{Root number method}

To prove that the set of rational points $\E(\Q)$ of an elliptic surface $\E$  is Zariski dense, it suffices to show that for infinitely many $t\in\P^1(\Q)$ the fibre $\E_t$ is an elliptic curve with positive Mordell-Weil rank, cf. \cite[Lem. 7.4]{SVL}.
For this reason, it is useful to study the \emph{root number}, denoted by $W(E)\in\{\pm1\}$. 
The root number $W(E)$ is conjecturally equal to the parity $(-1)^{rk(E)}$ of the Mordell-Weil rank\footnote{The parity conjecture $W(E)=(-1)^{rkE(\Q)}$ is equivalent to the congruence modulo $2$ of the algebraic rank with the analytic rank (the order of annulation of the $L$-function at $s=1$), thus it is a weakening of the Birch and Swinnerton-Dyer conjecture which predicts the equality of the two ranks, cf. \cite{DokchitserParity}.}.
Under this Parity Conjecture, it is thus sufficient to find infinitely many fibres with negative root number in order to prove the Zariski density.

Building on ideas of \cite{Manduchi}, \cite{Helfgott} and \cite{VA}, the first author \cite{Desjardins1} confirms (conditionally\footnote{This result depends on the parity conjecture and two more analytic number theory conjectures}) that non-isotrivial elliptic surfaces have a dense set of rational point through the study of the variation of the root number. 
For isotrivial elliptic surfaces, it can happen that a family has a constant root number, but \cite{Desjardins2} proves (unconditionally) that they have a dense set of rational points given that the surface has $j$-invariant $j\not=0$. If $W(E_t)=+1$ for all $t\in\P^1(\Q)$ and that the $j$-invariant is $0$, then it is not yet known whether the rational points are Zariski dense. 

A rational elliptic surface with $j=0$ has a Weierstrass equation $y^2=x^3+G(t)$ for some $G\in\mathbb{Q}[t]$ with $\deg G\leq6$; the difficult case correspond to a Del Pezzo surface of degree 1, i.e. when $\deg G = 5$ or $6$. For these V\'arilly-Alvarado \cite[Theorem 2.1]{VA} proved that the root number varies, provided that $G$ has a irreducible factor $G_i$ such that $\sqrt{-3}\not\in\mathbb{Q}[t]/(G_i)$. It is therefore natural to look at polynomials not satisfying this condition. 

The following proposition holds true: clearly, many of the elliptic surfaces covered in Theorem \ref{thmdescription} fall in this pattern. 
\begin{proposition}\label{prop:poly_embed}
	A square free polynomial $G(t)=c\prod_i{f_i}$, $f_i\in\mathbb{Z}[t]$ irreducible, satisfies the condition $$\mu_3\subset\Q[t]/f_i(t)\quad \text{for all }i,$$ (where $\mu_3$ is the group of third roots of unity) if and only if there exist non-zero polynomials $G_1,G_2\in\Z[t]$ and a constant $C\in\Z$ such that:
	$$G(t)=C(3G_1(t)^2+G_2(t)^2).$$
\end{proposition}
\begin{proof}
	Suppose there exist two polynomials $G_1,G_2\in\Z[t]$ and a constant $C\in\Z$  such that $G(t)=C(3G_1(t)^2+G_2(t)^2)$. Let $\alpha\in\overline{\mathbb{Q}}$ denote a root of a factor $f=f_i$. Then $G_1(\alpha)\neq 0$, otherwise it would follow that $f|G_2$ and $f^2|G$, contradicting the assumption on $G$. 
	Hence $-3=\left(\frac{G_2(\alpha)}{G_1(\alpha)}\right)^2$ and $\mu_3\subset \mathbb{Q}(\alpha)$.
	
	Let $G(t)=c\prod_{i}f_i$, where $f_i$ are irreducible polynomials such that the fields $K_i=\mathbb{Q}[t]/f_i$ contain $\mu_3$. In particular, fields $K_i$ have even degree $2d$ over $\mathbb{Q}$ and $\sqrt{-3}\in K_i$. We fix one index $i$ for now. Let $f=f_i$ and $K=K_i$. We denote by $\alpha$ a certain root $\alpha_i$ of $f_i$. 
	
	\begin{lemma}\label{lem:PQ_lemma}
		There exist two polynomials $P$ and $Q$ with integer coefficients such that $\deg P\leq d$ and $\deg Q\leq d-1$ and  
		\begin{equation}\label{eq:sq_m3}
			\sqrt{-3}=\frac{P(\alpha)}{Q(\alpha)}.
		\end{equation}
	\end{lemma}
	
	\begin{proof}[Proof of Lemma \ref{lem:PQ_lemma}]
		Suppose we have an element $\delta=\sum_{k=0}^{2d-1}c_k \alpha^k$ in a number field $K=\mathbb{Q}(\alpha)$ of even degree $2d$. Let $\beta=\sum_{k=0}^{d-1} b_k \alpha^k$. Then $$\delta\beta = \sum_{k=0}^{2d-1}\ell_{k}(b_0,\ldots,b_{d-1})\alpha^k$$ where the expressions $\ell_{k}$ are linear forms in $b_0,\ldots, b_{d-1}$ with coefficients in $\mathbb{Z}$ which depend on $\{c_k\}$.
		The linear system $$\ell_{k}(b_0,\ldots,b_{d-1})=0,\quad d+1\leq k\leq 2d-1$$
		has $d$ variables and $d-1$ equations, so has a nontrivial solution, which provides the coefficients $a_k=\ell_{k}(b_0,\ldots, b_{d-1})$. 
	\end{proof}
	It follows from Lemma \ref{lem:PQ_lemma} that $P(\alpha)^2+3Q(\alpha)^2=0$ and thus $f$ divides the polynomial $P^2+3Q^2$, so in fact they are proportional due to degree conditions.
	Hence, for each $i$ we have $f_i=c_i (P_i^2+3Q_i^2)$ and that concludes the theorem.
	
\end{proof}

\begin{remark}
	The polynomials such that $\delta=P(\alpha)/Q(\alpha)$ with $\deg P\leq d$, $\deg Q\leq d-1$ are not unique but the fraction $P/Q$ is. Indeed, if $P(\alpha)/Q(\alpha)=\delta=R(\alpha)/S(\alpha)$, then $(PS-RQ)(\alpha)=0$ and since the degree of $PS-RQ$ is smaller then $\deg f$, it follows that $PS-RQ=0$, hence $P/Q=R/S\in\mathbb{Q}(t)$.
\end{remark}

Family $\E_{A,B}$ coincides with $\E_{G}$ up to linear change of variables if and only if $G_1$ and $G_2$ are coprime polynomials with $\max(\deg(G_1),\deg(G_1))=3$ and $G_1,G_2$ are constant times a cube and $A=3c\cdot a^2$ and $B=c\cdot b^2$ ($a,b,c\in\Z$ constants such that $a$ and $b$ are coprime). 
In that case the family forms a set of sextic twists for which it is not guaranteed that the root number of a fibre $\E_t$ varies when $t$ varies through $t\in\P(\Q)$. In a previous paper \cite[Theorem 6.1]{Desjardins2}, the first author gives the precise conditions on $a$, $b$ and $c$ for which the root number takes the same values on every fibre.

We apply the algorithm from Theorem \ref{thmdescription} to the elliptic surface with equation $\E:y^2=x^3+3ca^2t^6+cb^2$, where $a,b,c\in\Z\setminus\{0\}$ and $gcd(a,b)=1$. According to our decision algorithm: 
\begin{itemize}
	\item the generic rank $r_{\E}$ is $2$ if $c$ or $-3c$ (resp. $3c$ or $-c$) is a square and $4AB$ and $A$ (resp. $4AB$ and $B$) are cubes;
	\item the rank $r_{\E}$ is $1$ if $c$ or $-3c$ (resp. $3c$ or $-c$) is a square and either $4AB$ or $A$ (resp. $4AB$ or $B$) are cubes;
	\item the rank $r_{\E}$ is $0$ otherwise.
\end{itemize}

This leads to the following result:

\begin{corollary}\label{corollaryconstantrootnumber}
	Let $\E$ be the elliptic surface given by the equation $y^2=x^3+At^6+B$ where $A=3ca^2$ and $B=cb^2$ for some $a,b,c\in\Z\setminus\{0\}$ with $gcd(a,b)=1$. Then $\E$ has generic rank $r_\E=2$ if there exist some $\alpha,\beta\in\Z$ satisfying one of the following:
	\begin{itemize}
		\item $A=3^3\alpha^6$ and $B=2^4\beta^6$ (or $A$ and $B$ switched);
		\item $A=-\alpha^6$ and $B=-2^43^3\beta^6$ (or $A$ and $B$ switched);
	\end{itemize}
	Suppose that it is not the case. Then $\E$ has generic rank $r_\E=1$ if there exist some non-zero $p,q\in\Z$ coprime and not divisible by 2 and 3, such that $A$ and $B$, written up to sixth power representative and possibly switched in the equation \footnote{Switched means that we look at $y^2=x^3+Bt^6+A$.}, are among
	\begin{itemize}
		\item $A=3^3$, $B=p^2$;
		\item $A=-1$, $B= -3p^2$
		\item $A$ and $B$ have the same sign and appear in Table \ref{listrank1surfaces}
	\end{itemize}
	
	\begin{table}
		\begin{center}
			\begin{tabular}{c|c|c}
				$\vert A\vert$ & $\vert B\vert$ \\
				\hline
				$3p^2q^4$ &  $2^43^2p^4q^2$ \\
				$3^3p^2q^4$ &  $2^4p^4q^2$ \\
				$3^5p^2q^4$ &  $2^43^4p^4q^2$  \\
				$2^23p^2q^4$ &  $2^23^2p^4q^2$ \\
				$2^23^3p^2q^4$ &  $2^2p^4q^2$ \\
				$2^23^5p^2q^4$ &  $2^23^4p^4q^2$ \\
				$2^43p^2q^4$ &  $3^2p^4q^2$ \\			
				$2^43^3p^2q^4$ &  $p^4q^2$ \\
				$2^43^5p^2q^4$ &  $3^4p^4q^2$ \\
			\end{tabular}
		\end{center}
		\caption{Possibilities for $\vert A\vert$ and $\vert B\vert$ written up to sixth power representatives. We assume $p,q$ are cube-free, coprime and coprime to $6$.}\label{listrank1surfaces}
	\end{table}
	
	In all the other cases of $A$ and $B$, the generic rank is $r_\E=0$.
\end{corollary}

\begin{proof}
	The proof is based on the decision algorithm and is purely combinatorial. 
	
	Suppose either that $c$ or $-3c$ is a square, say $c=\gamma^2$. Then $A=3(\gamma a)^2$ and $B=(\gamma b)^2$. If $c=-3\gamma^2$, then $A=-(3\gamma a)^2$ and $B=-3(\gamma b)^2$. In both cases, we take the path $YES$ , $NO$ in Figure \ref{fig:Diagram_0}, the initial point of the procedure. This leads us to continue the decision algorithm on Figure \ref{fig:Subroutine_1}. 
	
	Suppose that $c=\gamma^2$. If $4AB$ is a cube, then $4AB=2^2 3\gamma^4 a^2b^2=\delta^3$. 
	In that case, we take the branch $YES$, and so :
	\begin{itemize}
		\item If $A$ if a cube, we take the branch $YES$ and thus $r_\E=2$.
		Then we obtain that $A=3^3\alpha^6$ and $B=2^4\beta^6$. Those computations are very explicit and we emphasise that we do them separately for valuations at primes $2$, $3$ and $p\geq 5$.
		
		\item If $A$ is not a cube, we take the branch $NO$ and thus $r_\E=1$. 
		Then $A$ can take any value among $2^{2e}3^{2f+1}p^2q^4\alpha^6$ with $e,f\in\{0,1,2\}$ as soon as it is not $e=0$, $f=1$ and $p=q$ - in which case it is a cube. 
		These values are listed in Table \ref{listrank1surfaces}. (We have $B=2^{4-2e}3^{2(1+2f)}p^4 q^2\beta^6$.)
	\end{itemize}
	Now, if $4AB$ is not a cube, we follow the path $NO$ in the first step of Figure \ref{fig:Subroutine_1}.
	\begin{itemize}
		\item If $A$ is a cube, we take the branch $YES$ and thus $r_\E=1$. 
		This happens if $A=3^{3}\alpha^6$.
		The possibilities for $B$ will be among those such that
		$B=2^{2e}3^{2f}p^2\beta^6$ for all choices of $e,f\in\{0,1,2\}$ and $p\in\Z$ except the case $e=0$, $f=1$ and $p=\alpha^3$ - because in that case $4AB$ is a cube and hence we are not anymore on the right branch of the diagram: we rather have $r_\E=2$.
		\item If $A$ is not a cube, we take the branch $NO$ and thus $r_\E=0$. 
	\end{itemize}
	We find the possibility for $A$ and $B$ when $c=-3\gamma^2$ in a very similar way, using the same path in the routine.
	
	The cases for $-c$ square or $3c$ square can be obtained in a similar way - this time we take the path $NO$, $YES$ in Figure \ref{fig:Diagram_0} and we run through the different path of Figure \ref{fig:Subroutine_2}.
\end{proof}

\subsection{New examples proved with the Main theorem}\label{sec:new_examples}

\begin{theorem}\label{thm:new_examples}
	Every elliptic surface given by an equation of the form $\E:y^2=x^3+c(3a^2t^6+b^2)$ with generic rank 2 has constant root number on their fibres: $W(\E_t)=+1$ for all $t\in\Q$.
	
	The complete list of elliptic surfaces with a generic fibre of the form $y^2=x^3+At^6+B$ with constant root number on its fibres and generic rank 1 is given by Table \ref{listeconstantrootnumberrank1}. The table also states when $W(\E_t)=-1$ for all $t\in\Q$ and when $W(\E_t)=+1$ for all $t\in\Q$.
\end{theorem}

\begin{table}
	\begin{center}
		\begin{tabular}{c|c|c|c}
			$A$ & $B$ & Additional condition &Root number of the fibres\\ \hline
			$3^3p^2q^4$&$2^4p^4q^2$&$p^2q^4\equiv4\mod9$&$(-1)^{\sigma+1}$\\
			&			&$p^4q^2\equiv1\mod9$&$(-1)^{\sigma}$\\ \hline
			$3^3$&$2^43^2p^2$ & $p^2\equiv7\mod9$&$+1$\\ \hline
			$-2^43^3p^2q^4$&$-p^4q^2$ &$p^2q^4\equiv1\mod9$&$(-1)^{\sigma+1}$ \\ 
			&&$p^4q^2\equiv4\mod9$&$(-1)^\sigma$\\ \hline
			$-2^43^3p^2$&$-1$&$p^2\equiv1\mod9$&$+1$\\ \hline
			$-2^43^5p^2$&$-1$&any $p,q$&$+1$\\ \hline
		\end{tabular}
	\end{center}
	\caption{List of elliptic surfaces of the form $y^2=x^3+At^6+B$ with constant root number and of generic rank 1. $A$ and $B$ written up to sixth power representatives (and up to switching their role in the equation), and $\gcd(p,q)=1$. Integer $\sigma$ denotes the cardinality of the set $\{\text{prime factors $p_i$ of $pq$: $v_{p_i}(p^2q^4) \equiv 2,4\mod6$, $p_i\equiv2\mod3$}\}$.}
	\label{listeconstantrootnumberrank1}
\end{table}

\begin{remark}
	For those of the surfaces listed above with $W(\E_t)=+1$ for all $t\in\Q$, Corollary \ref{corollaryconstantrootnumber}  the first unconditional proof of the Zariski density of the rational points.
\end{remark}

\begin{remark} Unfortunately, there are many surfaces for which neither our result nor the root number method is decisive. These cases are those with simultaneously 
	\begin{enumerate}
		\item all root numbers are equal to $+1$ and
		\item $rk \E_{A,B}(\Q(t))=0$.
	\end{enumerate}
	We refer to Section \ref{sec:density_gen_rank_0} for a longer discussion on these cases. Moreover, we prove the density on some examples with constant root number with different techniques: \cite[Example 5.1]{VA} (also covered by our Theorem \ref{thmdescription}) and by first author and R. Winter \cite{DW}.
\end{remark}

\begin{proof}[Proof of Theorem \ref{thm:new_examples}]
	This is a consequence of Corollary \ref{corollaryconstantrootnumber} combined with \cite[Theorem 6.1.]{Desjardins2}.
	
	Let $\E$ be an elliptic surface given by the equation $y^2=x^3+c(3a^2t^6+b^2)$.  Write $t=\frac{m}{n}$, where $m,n\in\Z\times\Z_{>0}$ are coprime integers. Then the root number of a fibre is obtained from the formula (\cite[Prop. 4.8]{VA}): 
	\begin{equation}\label{formulern}W(\E_t)=-R(t)\prod_{p^2\mid F(m,n)\atop p\geq5}\begin{cases}
			1&\text{if }v_p(F(m,n))\equiv0,1,3,5\mod6\\
			\left(\frac{-3}{p}\right)&\text{if }v_p(F(m,n))\equiv2,4\mod6\end{cases}
	\end{equation}
	where
	$$R(t)=W_2(\E_t)\left(\frac{-1}{F(m,n)_{(2)}}\right)W_3(\E_t)(-1)^{v_3(F(m,n))},$$
	where $F(m,n)=c(3a^2m^6+b^2n^6)$. For a positive integer $\alpha$, we denote by $\alpha_{(p)}$ the integer such that $\alpha=p^{v_p(\alpha)}\alpha_{(p)}$. The product over $p^2\mid F(m,n)$ equals $(-1)^\sigma$, where $\sigma=\#\{p\mid c \ : \ p\equiv2\mod3\}$ since our choice of $F(m,n)$ has the property that whenever $p\mid F(m,n)$ and $p\nmid c$ for a $p\geq5$, we have
	$$F(m,n)=c(3a^2m^6+b^2n^6)\equiv0\mod p$$ 
	and thus $$\left(\frac{bn^3}{am^3}\right)^2\equiv-3\mod p,$$
	forcing $\left(\frac{-3}{p}\right)=+1$.
	Let us denote the function $\omega_2(t)=W_2(\E_t)\left(\frac{-1}{F(m,n)_{(2)}}\right)$ and $\omega_3(t)=W_3(\E_t)(-1)^{v_3(F(m,n))}$. Observe that we have:
	$$W(\E_t)=(-1)^{\sigma+1}\omega_2(t)\omega_3(t).$$
	The values of $A$ and $B$ for which the functions $\omega_2(t)$ and $\omega_3(t)$ are constant are listed in \cite[Lemma A.1, Lemma A.2]{Desjardins2}. In the first case, this depends on the quantities $v_2(a),v_2(b),v_2(c)$ and $c_{(2)}\mod 4$. In the second case, this depends on the quantities $v_3(a),v_3(b),v_3(c) \mod 6$ and $c_{(3)},a_{(3)}^2,b_{(3)}^2 \mod 9$.
	Let us illustrate the computation that needs to be done with a specific example (doing all cases exhaustively would be too long, and not necessary). 
	
	Let $\alpha,\beta\in\Z$, we study the surface given by the equation $y^2=x^3+3^3\alpha^6t^6+2^4\beta^6$.
	
	We have $v_2(a)=0$, $v_2(b)=2$ and $v_2(c)=0$. Moreover $c_{(2)}\equiv1\mod4$. It follows that $\omega_2(t)=+1$ for all $t\in\Q$ according to \cite[Table 3]{Desjardins2}.
	We have $v_3(a)=1$, $v_3(b)=0$ and $v_3(c)=0$. Moreover $c_{(3)}a_{(3)}^2\equiv1\mod9$. Thus, $\omega_3(t)=-1$ for all $t\in\Q$ according to \cite[Table 2]{Desjardins2}. Hence, for any $t\in\Q$ such that $\E_t$ is non-singular, the root number of this fibre is $$W(\E_t)=-(-1)(+1)=+1.$$
\end{proof}

\section{The structure of the orbits}\label{sec:structure_of_the_orbits}

\subsection{A basis of the Mordell-Weil group: the eight generators}\label{sec:preliminaries_on_MW_groups}

Let $k$ be a field of characteristic different from $2$ and $3$ and let $A,B\in k$ be two non-zero elements of $k$. 
Let $E_{A,B}:y^2 z=x^3+(At^6+B)z^3$ denote an elliptic curve over a field of rational functions $k(t)$.

Let $\mathcal{E}_{A,B}$ denote an elliptic surface attached to $E_{A,B}$. It is a smooth projective surface with a morphism $\pi:\mathcal{E}_{A,B}\rightarrow\mathbb{P}^{1}$ which is the natural projection $(x,y,z,t)\mapsto t$. We denote by $\mathcal{O}$ the zero section $\mathcal{O}:\mathbb{P}^{1}\rightarrow \mathcal{E}_{A,B}$ obtained by extending the zero point on $E_{A,B}$.

To each point $P\in E_{A,B}(\overline{k}(t))$ we associate a section $\sigma_{P}:\mathbb{P}^{1}\rightarrow\mathcal{E}_{A,B}$ and denote by $\overline{P}$ the $(-1)$-curve which is the image of $\sigma_{P}$. We denote by $\overline{P}.\overline{Q}$ the intersection number of $\overline{P}$ against $\overline{Q}$ on $\mathcal{E}_{A,B}$.
The intersection $\overline{P}.\overline{Q}$ (for $P\neq Q$) is computed as a sum of local intersection numbers (\cite[III \S 7, \S 9]{Silvermand_advanced})
\[\overline{P}.\overline{Q} = \sum_{t\in \mathbb{P}^1}(P,Q)_{t}.\]

In fact, the translation by point map on the generic fibre $E_{A,B}$ extends to an automorphism of $\mathcal{E}_{A,B}$, \cite[III Prop. 9.1]{Silvermand_advanced}, hence $\overline{P}.\overline{Q}=\overline{P-Q}.\overline{O}$. 

The intersection number $\overline{P}.\overline{O}$ is computed by the formula
\[\overline{P}.\overline{O}=\deg(q(t))+\delta\]
where $x(P)=\frac{p(t)}{q(t)^2}$ is the $x$-coordinate of $P$ given in terms of two coprime polynomials $p,q$. A non-negative integer $\delta$ is computed in the integral model of $E_{A,B}$ at $\infty$ where $t=1/s$ and the local model is $y^2=x^3+A+Bs^6$. Number $\delta$ satisfies the condition $\frac{p(1/s)}{q(1/s)^2}s^2=\frac{\tilde{p}(s)}{\tilde{q}(s)^2 s^{2\delta}}$ where $\tilde{p}$ and $\tilde{q}$ are coprime to $s$ and $\delta=\max\{0,\deg p/2-\deg q-1\}$.

Group $E_{A,B}(\overline{k}(t))$ is equipped with the height pairing (\cite[III \S 4, \S 9]{Silvermand_advanced}) defined for two given points $P$, $Q$ by the formula (\cite{Shioda_MW})
\[\langle P,Q\rangle = 1+\overline{P}.\overline{O}+\overline{Q}.\overline{O}-\overline{P}.\overline{Q}.\]
We denote by $\langle P,P\rangle$ the height of the point $P$ and have the simplified formula $\langle P,P\rangle = 2+2\overline{P}.\overline{O}\geq 0$. In particular, it implies that we have no non-trivial torsion points and $\langle P,P\rangle=2$ if and only if $\overline{P}.\overline{O}=0.$ The latter condition means that the $x$-coordinate of $P$ is a polynomial in $t$ of degree at most $2$.
Group $E_{A,B}(\overline{k}(t)$ with the pairing $\langle\cdot,\cdot\rangle$ forms a positive definite lattice. 
This is a special case of the general theory of the Mordell-Weil lattices, cf. \cite{Shioda_MW}.

We use throughout the rest of the paper the symbols of standard lattices $A_n$, $D_n$ and $E_n$ which correspond to Dynkin diagrams with the same notation, cf. \cite[Chap. 6]{Bourbaki_Lie}.
\begin{proposition}\label{prop:pt_structure_on_ell_surf}
	Let $A,B\in k\setminus\{0\}$. The group $E_{A,B}(\overline{k}(t))$ has rank $8$ and no non-trivial torsion elements. The generators are contained in the set of points $P$ of the form
	\[P=(a t^2+bt+c,a_1 t^3+a_2 t^2+a_3 t+a_4),\quad a,b,c,a_1,a_2,a_3,a_4\in\overline{k}.\]
\end{proposition}
\begin{proof}
	The equation $E_{A,B}$ defines a smooth cubic over $\overline{k}(t)$ for $At^6+B$ non-zero. When both $A,B$ are non-zero the cubic equation $E_{A,B}$ is not isomorphic to a constant  cubic defined over $k$. 
	From \cite[10.14]{Shioda_MW} it follows that $\mathcal{E}_{A,B}$ is a rational elliptic surface.
	The discriminant of $E_{A,B}$ equals $-432 \left(A t^6+B\right)^2$. For $A,B$ non-zero the polynomial $A t^6+B$ is separable, 
	hence from the Tate algorithm \cite{Tate_algorithm}, \cite[IV.9.4]{Silvermand_advanced} it follows that the equation is minimal at all finite places of $\overline{k}(t)$ and at the places corresponding to the solutions of $At^6+B=0$ the reduction is of type $II$.
	To analyse the model at infinity $1/t$ we apply the change of coordinates $s=1/t$ and compute the integral model $y^2=x^3+A+Bs^6$ which is smooth at $s=0$ for $A\neq 0$.
	
	From the Shioda-Tate formula \cite[Thm. 7.4]{Shioda_MW} and \cite[Lem. 10.1]{Shioda_MW} it follows that the rank of $E_{A,B}(\overline{k}(t))$ is $8$. Classification of rational elliptic surfaces by Oguiso-Shioda \cite{Oguiso_Shioda} implies that the generators of the group $E_{A,B}(\overline{k}(t))$ have height $2$ and there is no non-zero torsion point. The group $E_{A,B}(\overline{k}(t))$ with the height pairing $\langle\cdot,\cdot\rangle$ forms an integral lattice of type $E_{8}$. Since the height of each generator $P$ is $2$, we obtain that in the model $E_{A,B}$ the coordinates of the point $P$ are polynomials in $\overline{k}[t]$ of degrees $2$ and $3$, respectively.
\end{proof}

\begin{proposition}\label{proposition:orbits}
	Let $A,B\in\mathbb{Q}^{\times}$. The subset of $E_{A,B}(\overline{\mathbb{Q}}(t))$ of elements of height $2$ decomposes into $8$ disjoint subsets $O_{i}$ for $i=1,\ldots,8$. Each set $O_{i}$ is stable under the action of the absolute Galois group $G_{\mathbb{Q}}=\Gal(\overline{\mathbb{Q}}/\mathbb{Q})$ on the coordinates of the points. 
	
	For $A,B$ sufficiently generic the orbits $O_{i}$ do not decompose into smaller $G_{\mathbb{Q}}$-stable subsets.
\end{proposition}
\begin{remark}
	The term ''sufficiently generic'' refers to the detailed decision algorithm described in Section \ref{sec:proof_main_theorem}.
\end{remark}
\begin{proof}
	Steps of the algorithm:
	\begin{itemize}
		\item[1.] Form an ideal $I=(c_i)_{i=0}^{6}$ in the polynomial ring $$S[a,b,c,a_1,a_2,a_3,a_4]$$
		where $S=\mathbb{Q}(A,B)$, $x=x(P)$, $y=y(P)$ and $$y^2-(x^3+At^6+B)=\sum_{i=0}^{6}c_i t^i.$$ We have
		\begin{align*}
			I&=(a_4^2-B-c^3,2 a_3 a_4-3 b c^2,-3 a c^2+a_3^2+2 a_2 a_4-3 b^2 c,\\
			&-6 a b c+2 a_2 a_3+2 a_1 a_4-b^3,\\
			&-3 a^2 c-3 a b^2+a_2^2+2 a_1 a_3,2 a_1 a_2-3 a^2 b,-a^3+a_1^2-A).
		\end{align*}		
		\item[2.] Compute the elimination ideal $J$ in $S[a,b,c]$ with respect to $a,b,c$.
		\item[3.] Compute the primary decomposition $\bigcap_{i} J_{i}$ of the ideal $J$.
		\item[4.] Compute the Groebner basis $\{b_i\}$ with respect to the standard lexicographical order of the ideals spanned by $J_i\cup I$.
		\item[5.] Form a scheme $O_i$ which is the vanishing locus of the system $\{b_i\}$.
	\end{itemize}
	On the output we obtain the defining equations of the components $O_{i}$ computed in the step $5$. There are exactly eight of the them, so $S_{red}=\amalg_{i=1}^{8}O_i$ where each scheme $O_i$ is of dimension $0$, reduced and irreducible over $\mathbb{Q}(A,B)$. For $1\leq i\leq 8$ the degree of the scheme $O_i$ equals $6, 6, 12,$ $18, 18,$ $36, 36,$ $108$, respectively.
\end{proof}
In the following paragraphs we discuss the structure of the smallest orbits $O_{1}$, $O_{2}$ and $O_{3}$. The main result of this section is a proof that there exist $8$ linearly independent points in those three orbits. Therefore it is enough to study those $8$ points in order to determine the rank of the group $E_{A,B}(\mathbb{Q}(t))$ (since there are no torsion points).

We denote by $\mathcal{P}(O)$ the set of points in $E(\overline{\mathbb{Q}}(t))$ which correspond to the $O(\overline{\mathbb{Q}})$-points of the scheme $O$ in $S_{red}$. By abuse of notation we also say that the set $O$ contains a point $P$, if $P\in \mathcal{P}(O)$.
Let $\mathcal{L}(O)$ denote the $\mathbb{Z}$-span of the points in the set $O$.

\begin{remark}
	To verify that for each choice of non-zero elements $A,B$ the statements below hold (there is no degeneration) we have used a computer and the package Magma. The computation essentially reduces to a calculation similar to the one presented in the proof of Proposition \ref{prop:Finite_index_sublattice}. 
	In fact, for each proof below we verify the relations between points over the function field extension $\overline{\mathbb{Q}}(A^{1/6},B^{1/6})$ of $\overline{\mathbb{Q}}(A,B)$ and check that the lattice produced from the points in the orbit (and the Gram matrix of a given basis) are the same for every choice of nonzero $A,B$. The verification code is available on the website of one of the authors \cite{Magma_code}.
\end{remark}

\begin{remark}
	To simplify our notation we have adopted the following convention. For a given non-zero rational number $A\in\mathbb{Q}$, we find a polynomial factorization $x^6-A=\prod_{i}p_i(x)$, where $p_i(x)\in\mathbb{Q}[x]$ are monic, irreducible over $\mathbb{Q}$ and ordered by the degree: $\deg p_1\leq \deg p_2\leq \cdots $. We denote by $A^{1/6}$ any fixed root in $\overline{\mathbb{Q}}$ of the polynomial $p_1(x)$. We define in analogy the symbol $B^{1/6}$ for any non-zero rational number $B\in\mathbb{Q}$. 
	
	We denote by $\zeta_3$ an arbitrarily fixed root in $\overline{\mathbb{Q}}$ of the polynomial $x^2+x+1$ and we define $\sqrt{-3}$ to be $1+2\zeta_3$. 
	
	We denote by $2^{1/3}$ an arbitrarily fixed root in $\overline{\mathbb{Q}}$ of the polynomial $x^3-2$. If either $(A^{1/6})^2$ or  $(B^{1/6})^2$ is a root of $x^3-2$, we choose $2^{1/3}$ to be equal to one of these values.
	
	In the sections below, we use expressions of the form $\sqrt[k]{c A^{e}B^{f}}$ with $k\in \{2,3,6\}$ and $e,f\in \{0,1\}$. We define $\sqrt[k]{c A^{e}B^{f}}$ as $\sqrt[k]{c}\cdot (A^{1/6})^{6e/k}\cdot (B^{1/6})^{6f/k}$. We also fix that $\sqrt[3]{-1}=-1$.
\end{remark}
\subsubsection{Orbit $O_1$} 
The orbit $O_1$ is defined by the following conditions
\begin{align*}
	a_1^2&=A,\ c^3=-B,\\
	a_2&=a_3=a_4=a=b=0
\end{align*}
and determines a list of the following points
\[\mathcal{P}(O_1) = \{(-\zeta_3^i \sqrt[3]{B},\pm \sqrt{A} t^3): i \in \{0,1,2\} \}.\]
The set $\mathcal{P}(O_1)$ consists of $6$ points, each of height $2$ (under assumption that $AB\neq 0$). A configuration of the curves in the orbit $O_{1}$ is described in Figure~\ref{fig:orbit_1}. 

\definecolor{dtsfsf}{rgb}{0.8274509803921568,0.1843137254901961,0.1843137254901961}
\definecolor{rvwvcq}{rgb}{0.08235294117647059,0.396078431372549,0.7529411764705882}
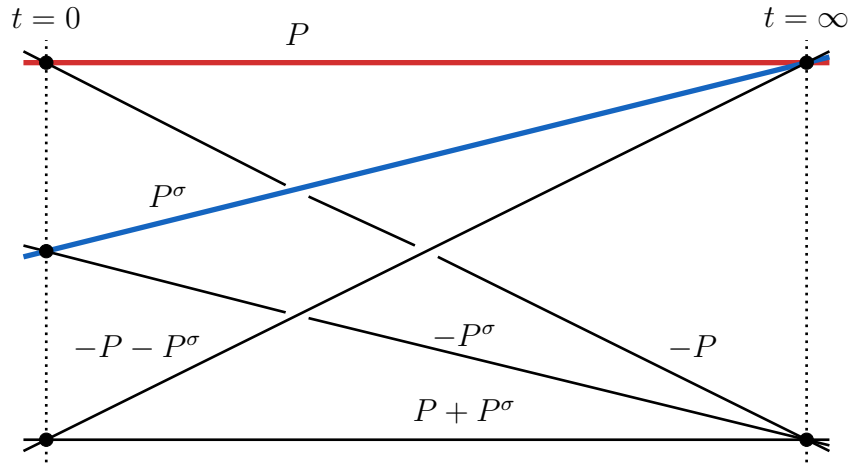
\begin{figure}
	\centering
	\begin{tikzpicture}[x=1cm,y=1cm]
		\newcommand{\yheight}{5.0}
		\newcommand{\xwidth}{10.0}
		\def\ratio{\yheight/\xwidth}
		\def\eps{0.15}
		\coordinate (LB) at (0,0); 
		\coordinate (RB) at (\xwidth,0); 
		\coordinate (LM) at (0,0.5*\yheight); 
		\coordinate (LT) at (0,\yheight); 
		\coordinate (RT) at (\xwidth,\yheight); 
		
		\draw [line width=2pt, color=dtsfsf](0-2*\eps,\yheight)--(\xwidth+2*\eps,\yheight); 
		\draw [line width=1pt](0-2*\eps,0)--(\xwidth+2*\eps,0); 
		
		\draw [line width=1pt, dotted](0,0-2*\eps)--(0,\yheight+2*\eps);
		\draw [line width=1pt, dotted](\xwidth,0-2*\eps)--(\xwidth,\yheight+2*\eps); 
		
		\path [line width=1pt,color=black] 
		(0-2*\eps,0.5*\yheight+\ratio*0.5*2*\eps) edge (0.33*\xwidth-\eps,0.33*\yheight+\ratio*0.5*\eps)
		(0.33*\xwidth-\eps,0.33*\yheight+\ratio*0.5*\eps) edge[white] (0.33*\xwidth+\eps,0.33*\yheight-\ratio*0.5*\eps)
		(0.33*\xwidth+\eps,0.33*\yheight-\ratio*0.5*\eps) edge (\xwidth+2*\eps,0-\ratio*0.5*2*\eps); 
		
		\path [line width=1pt,color=black] 
		(0-2*\eps,\yheight+\ratio*2*\eps) edge (0.33*\xwidth-\eps,0.66*\yheight+\ratio*\eps)
		(0.33*\xwidth-\eps,0.66*\yheight+\ratio*\eps) edge[white] (0.33*\xwidth+\eps,0.66*\yheight-\ratio*\eps)
		(0.33*\xwidth+\eps,0.66*\yheight-\ratio*\eps) edge (0.5*\xwidth-\eps,0.5*\yheight+\ratio*\eps)
		(0.5*\xwidth-\eps,0.5*\yheight+\ratio*\eps) edge[white] (0.5*\xwidth+\eps,0.5*\yheight-\ratio*\eps)
		(0.5*\xwidth+\eps,0.5*\yheight-\ratio*\eps) edge (\xwidth+2*\eps,0-\ratio*2*\eps); 
		
		\draw [line width=2pt,color=rvwvcq] (0-2*\eps,0.5*\yheight- \ratio*0.5*2*\eps)--(\xwidth+2*\eps,\yheight+\ratio*0.5*2*\eps); 
		
		\draw [line width=1pt,color=black] (0-2*\eps,-\ratio*2*\eps)--(\xwidth+2*\eps,\yheight+\ratio*2*\eps); 
		
		\draw [fill=black] (LM) circle (2.5pt); 
		\draw [fill=black] (RT) circle (2.5pt); 
		\draw [fill=black] (RB) circle (2.5pt); 
		\draw [fill=black] (LB) circle (2.5pt); 
		\draw [fill=black] (LT) circle (2.5pt); 
		
		\draw node at (0,\yheight+8*\ratio*\eps){$t=0$};
		\draw node at (\xwidth,\yheight+8*\ratio*\eps){$t=\infty$};
		\draw node at (0.33*\xwidth,\yheight+5*\ratio*\eps){$P$};
		\draw node at (0.16*\xwidth,0.57*\yheight+5*\ratio*\eps){$P^{\sigma}$};
		\draw node at (0.12*\xwidth,0.17*\yheight+5*\ratio*\eps){$-P-P^{\sigma}$};
		\draw node at (0.85*\xwidth,0.17*\yheight+5*\ratio*\eps){$-P$};
		\draw node at (0.55*\xwidth,0.20*\yheight+5*\ratio*\eps){$-P^{\sigma}$};
		\draw node at (0.55*\xwidth,0*\yheight+5*\ratio*\eps){$P+P^{\sigma}$};
		
	\end{tikzpicture}
	\caption{\label{fig:orbit_1} Configuration of the lines in $\mathcal{P}(O_1)$ on the surface.}
\end{figure}

Let $K$ denote the field $\mathbb{Q}(\sqrt{A},\sqrt[3]{B},\zeta_{3})$, $\zeta_3^3=1$ and $\zeta_3\neq 1$. Consider a map $\sigma_{K}: \sqrt[3]{B}\mapsto \zeta_{3}\sqrt[3]{B}$ which fixes the other generators of $K$. If it extends to an automorphism of $K$, it fixes the subfield $\mathbb{Q}(\sqrt{A},\zeta_3)$. 
\begin{remark}
	In the case when $A$ and $B$ are such that the corresponding automorphism does not exist we denote by $T^\sigma$ the point obtained from $T$ by substitution of $\zeta_{3}\sqrt[3]{\alpha}$ in place of $\sqrt[3]{\alpha}$.
\end{remark}

\begin{proposition}\label{prop:Point_P_properties}
	Let $A,B$ be non-zero elements in $\mathbb{Q}$. Then a point $P=(-\sqrt[3]{B},\sqrt{A} t^3)$ lies on the elliptic curve $E=E_{A,B}$ and
	\[P+P^{\sigma_{K}}+P^{\sigma_{K}^{2}}=0.\]
	Moreover, the height $\langle P,P\rangle$ equals $2$ and the height pairing matrix of the points $P,P^{\sigma_{K}}$ is 
	\[M:=\left(\begin{array}{cc}
		2 & -1 \\
		-1 & 2 \\
		
	\end{array}\right).\]
	In particular, the points $P,P^{\sigma_{K}}$ are linearly independent in $E(\overline{\mathbb{Q}}(t))$. Moreover, the $\mathbb{Z}$-linear span $\mathcal{L}(O_1)$ has rank $2$ and basis $P,P^{\sigma_{K}}$.
\end{proposition}
\begin{proof}
	The first equality follows from a direct computation. The height $\langle P,P\rangle$ of $P$ and the height $\langle P^{\sigma_{K}},P^{\sigma_{K}}\rangle$ of $P^{\sigma_{K}}$ is computed with the formula described in Section \ref{sec:preliminaries_on_MW_groups}.
	
	To compute the height pairing matrix $M$ we need only to check that the height of $P+P^{\sigma_{K}}$ is $2$ and the equality $\langle P,P^{\sigma_{K}}\rangle =-1$ follows from bilinearity of the pairing.
	The determinant of $M$ is non-zero, hence the points $P$ and $P^{\sigma_{K}}$ are linearly independent over $\mathbb{Z}$.
	Finally, observe that $\mathcal{P}(O_1)$ has $6$ elements which coincide with the elements of the set $\{\pm P,\pm P^{\sigma_{K}},\pm P \pm P^{\sigma_{K}}   \}.$
\end{proof}

\subsubsection{Orbit $O_2$} 
The orbit $O_2$ is defined by the following equations
\begin{align*}
	a_4^2&=B,a^3=-A,\\
	a_1&=a_2=a_3=b=c=0.
\end{align*}
Configuration of the curves in the orbit $O_{2}$ is is very similar to Figure \ref{fig:orbit_1}.
Let $K'$ denote the field $\mathbb{Q}(\sqrt{B},\sqrt[3]{A},\zeta_3)$ and let $\sigma_{K'}:\sqrt[3]{A}\mapsto \zeta_3\sqrt[3]{A}$, which fixes the other generators of $K'$, denote a map which if extended to an automorphism of $K'$ fixes the subfield $\mathbb{Q}(\sqrt{B},\zeta_3)$. We consider a point $Q=(-\sqrt[3]{A}t^2,\sqrt{B})$ that corresponds to a point on the orbit $O_2$. 
\begin{proposition}\label{prop:Point_Q_properties}
	Let $A, B$ be non-zero elements from $\mathbb{Q}$. Then the height pairing matrix of the points $Q, Q^{\sigma_{K'}}$ is $M$ from Proposition \ref{prop:Point_P_properties}. The following equality 
	\[Q+Q^{\sigma_{K'}}+Q^{\sigma_{K'}^{2}}=0\]
	holds and $\mathcal{L}(O_2)$ is of rank $2$ with basis $Q,Q^{\sigma_{K'}}$.
\end{proposition}
\begin{proof}
	We omit the proof since it is analogous to that of Proposition \ref{prop:Point_P_properties}.
\end{proof}

\subsubsection{Equations of $O_3$}
The orbit $O_3$ is defined by the following equation
\begin{align*}
	a_ 1^2 & = A,\ a_4^2 = B,\ 2 a_ 1 a_ 4 = b^3,\\
	a &= c = a_2 = a_3 = 0.
\end{align*}
Note that $b^6- 4 A B$ belongs to the ideal of $O_3$. A configuration of the points in the orbit $O_3$ is depicted in the Figure \ref{fig:orbit_3} with extra data in Table \ref{tab:table_orbit_3}.

Let $L$ denote the field $\mathbb{Q}(\sqrt{A},\sqrt{B},\zeta_3,\sqrt[3]{4AB})$. Let $\sigma_{L}$ be a map which maps $\sqrt[3]{4AB}\mapsto \zeta_3 \sqrt[3]{4AB}$ and fixes the other generators of $L$. If it extends to an automorphism of $L$, then it fixes the subfield $\mathbb{Q}(\sqrt{A},\sqrt{B},\zeta_3)$. Let $\tau_{L}$ denote a map such that $\tau_{L}(\sqrt{B})=-\sqrt{B}$. If it extends to an automorphism of $L$, then it fixes the field $\mathbb{Q}(\sqrt{A},\zeta_3,s)$.
\begin{remark}
	In the case when $A$ and $B$ are such that the corresponding automorphism does not exist we denote by $T^\tau$ the point obtained from $T$ by substitution of $-\sqrt{\alpha}$ in place of $\sqrt{\alpha}$.
\end{remark}
We consider the points 
\begin{align*}
	R&=(\frac{2\sqrt{A}\sqrt{B}}{\sqrt[3]{4AB}}t,\sqrt{A}t^3+\sqrt{B}),\\
	S&=(-\frac{2\sqrt{A}\sqrt{B}}{\sqrt[3]{4AB}}t,\sqrt{A}t^3-\sqrt{B})
\end{align*}
contained in $E_{A,B}(\overline{\mathbb{Q}}(t))$.

\begin{proposition}\label{prop:Point_RS_properties}
	Let $A, B$ be non-zero elements from $\mathbb{Q}$.
	Points $R$ and $S$ have height $2$ and they satisfy the following identities
	\[R+R^{\sigma_{L}}+R^{\sigma_{L}^2}=0\]
	and 
	\[S+S^{\sigma_{L}}+S^{\sigma_{L}^2}=0.\]
	The Gram matrix of the pairs $R,R^{\sigma_{L}}$ and $S,S^{\sigma_{L}}$ is $M$ from Proposition \ref{prop:Point_P_properties}.
	The linear span $\mathcal{L}(O_3)$ has rank $4$ and basis $R,R^{\sigma_{L}},S,S^{\sigma_{L}}$.
\end{proposition}
\begin{proof}
	We omit the proof since it is analogous to that of Proposition \ref{prop:Point_P_properties}.
\end{proof}

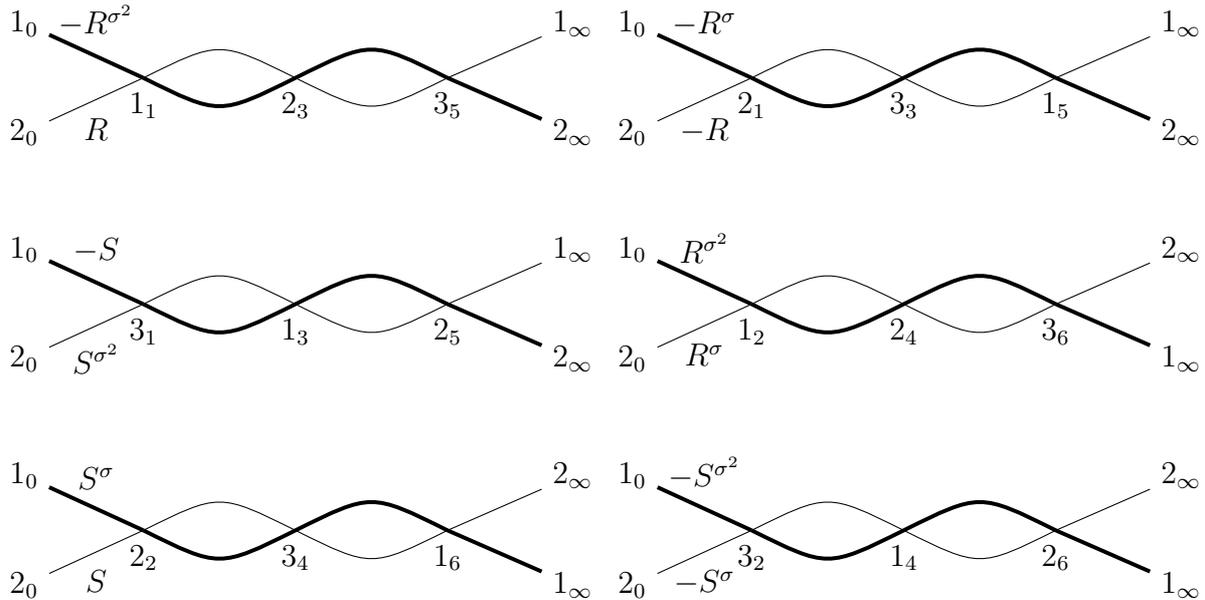
\begin{figure}[htb]
	\centering
	\begin{tikzpicture}[x=1cm,y=1cm]
		\orbitthree{1}{2}{-R^{\sigma^2}}{R}{1_1}{2_3}{3_5}{0}{0}
		
		\orbitthree{1}{2}{-R^{\sigma}}{-R}{2_1}{3_3}{1_5}{4}{0}
		
		\orbitthree{1}{2}{-S}{S^{\sigma^2}}{3_1}{1_3}{2_5}{0}{-1.5}
		
		\orbitthree{2}{1}{R^{\sigma^2}}{R^{\sigma}}{1_2}{2_4}{3_6}{4}{-1.5}
		
		\orbitthree{2}{1}{S^{\sigma}}{S}{2_2}{3_4}{1_6}{0}{-3}
		
		\orbitthree{2}{1}{-S^{\sigma^2}}{-S^{\sigma}}{3_2}{1_4}{2_6}{4}{-3}
	\end{tikzpicture}
	\caption{\label{fig:orbit_3} Configuration of the lines in $\mathcal{P}(O_3)$ on the surface.}
\end{figure}

\begin{table}[htb]
	\centering
	\begin{tabular}{c|c}
		Point label & description\\
		\hline
		$1_0 $ & $(0,\sqrt{B})$ at $t=0$\\
		$2_0 $ & $(0,-\sqrt{B})$ at $t=0$\\
		$1_\infty $ & $(0,\sqrt{A})$ at $t=\infty$\\
		$2_\infty $ & $(0,-\sqrt{A})$ at $t=\infty$\\
		$m_i$ where $0<i<\infty$ & $(\zeta_6^{2m-1}(2B)^{1/3},0)$ at $t=\zeta_6^{i-1}(\frac{A}{B})^{1/6}$
	\end{tabular}
	\caption{Description of the labels in Figure \ref{fig:orbit_3}}\label{tab:table_orbit_3}
\end{table}

\begin{proposition}\label{prop:Finite_index_sublattice}
	For $A$ and $B$ non-zero rational numbers a subgroup in $E(\overline{\mathbb{Q}}(t))$ spanned by the points $\{P,P^{\sigma_{K}}, Q, Q^{\sigma_{K'}}, R, R^{\sigma_{L},} S, S^{\sigma_{L}}\}$ has rank $8$ and the Gram matrix equal to the Kronecker product $I_{4}\otimes M$ where $I_4$ is the $4$ by $4$ identity matrix.
\end{proposition}
\begin{proof}
	Let $B_i$ denote the $i$-th element of the ordered tuple \[(P,P^{\sigma_{K}}, Q, Q^{\sigma_{K'}}, R, R^{\sigma_{L},} S, S^{\sigma_{L}}).\] We need to show that $\langle B_{i},B_{j}\rangle = 0$ for $i<j$ and $i+1\neq j$, $i,j \in \{1,\ldots,8\}$. From the bilinearity of the height pairing this is equivalent to showing $\langle B_{i}+B_{j},B_{i}+B_{j}\rangle =4$ or equivalently $\overline{B_{i}+B_{j}}.\overline{O}=1$. 
	To show that the curve $\overline{B_{i}+B_{j}}$ intersects the image of the zero section $\overline{O}$ exactly once we look at the $x$-coordinate of $\overline{B_{i}+B_{j}}$. For $t=\infty$ there is no intersection with
	$\overline{O}$ and for each pair of $(i,j)$ defined above there is exactly one intersection if $A,B$ are both non-zero. This last condition is verified in Magma code attached to this paper.
\end{proof}

\subsection{Further orbits}
Although it is not strictly necessary for our arguments to compute the precise structure of the orbits $O_i$ for $i\geq 4$ we include it for completeness. 
Some of those orbits could be used to provide an alternative proof of the theorem of Section \ref{sec:proof_main_theorem}. 
For each orbit $O_i$ we compute the rank and a basis of the linear span $\mathcal{L}(O_i)$. We also use some of the points from orbits $O_4, O_5,O_6$ and $O_7$ to conveniently describe bases of points in $E_{A,B}(\mathbb{Q}(t))$, cf. Section \ref{sec:rational_basis}.

\subsubsection{Orbits of size 18}
The orbit $O_4$ has size $18$ and is defined by the following equations
\begin{align*}
	a^3 &=-4A,\  c^3=-B,\ a_3^2=3ac^2,\\
	a_1 &=-\frac{1}{2B}a_3 a c^2,\ a_2=a_4=b=0.
\end{align*}
Let $U_{j,k,m}$ denote a point
\[\left(-\zeta _3^j \left(\sqrt[3]{4A} t^2+\sqrt[3]{B} \zeta _3^k\right),\frac{(-1)^{m+1}\sqrt[6]{A} \sqrt{-3} \left( \sqrt[3]{4A} t^2+2 \sqrt[3]{B} \zeta _3^k\right)t}{2^{2/3}}\right)\]
where $j,k$ belong to the set $\{0,1,2\}$ and $m$ equals $1$ or $2$.
\begin{proposition}\label{prop:Orbit_4_structure}
	Let $A, B$ be non-zero elements from $\mathbb{Q}$. It follows that 
	\[\mathcal{P}(O_4)=\{U_{j,k,m}: j,k \in \{0,1,2\}, m\in\{1,2\} \}.\]
	The linear span $\mathcal{L}(O_4)$ has rank $4$ and is generated by the points $U_{0,0,1},$ $U_{0,1,1},$ $U_{1,0,1},$ $U_{1,1,1}$. The lattice $(\mathcal{L}(O_4),\langle\cdot,\cdot\rangle)$ is isometric to the $D_4$ lattice.
\end{proposition}
\begin{proof}
	The points in the set $\mathcal{P}(O_4)$ satisfy the following relations
	\[U_{j,k,1}=-U_{j,k,2}\quad\textrm{for any $j,k$,}\]
	\[U_{j,0,1}+U_{j,1,1}+U_{j,2,1}=0\quad\textrm{for any $j$,}\]
	\[U_{0,k,1}+U_{0,k,1}+U_{2,k,1}=0\quad\textrm{for any $k$.}\]
	Since $\mathcal{P}(O_4)$ contains $18$ points, the relations above imply that there are at most $4$ points independent over $\mathbb{Z}$ in $\mathcal{L}(O_4)$.  
	We check by a direct computation that the points indicated in the proposition are linearly independent. The Gram matrix of their height pairing is conjugate to the Gram matrix of the lattice $D_{4}$.
\end{proof}

The orbit $O_5$ has size $18$ and is defined by the equations
\begin{align*}
	a^3&=-A,\  c^3=-4B,\ a_4^2=-3B,\\
	a_2&=-\frac{1}{2B}a_4 a c^2,\ a_1=a_3=b=0.
\end{align*}

Let $V_{j,k,m}$ denote a point

\[V_{j,k,m}=\left(-\zeta _3^j \left(\sqrt[3]{A}\zeta _3^k t^2 + \sqrt[3]{4B}\right),(-1)^{m+1}\sqrt[6]{B} \sqrt{-3} \left(\sqrt[3]{2 A}\zeta _3^k t^2 +\sqrt[3]{B}\right)\right).\]

\begin{proposition}
	Let $A, B$ be non-zero elements from $\mathbb{Q}$. It follows that 
	\[\mathcal{P}(O_5)=\{V_{j,k,m}: j,k \in \{0,1,2\}, m\in\{1,2\} \}.\]
	The linear span $\mathcal{L}(O_4)$ has rank $4$ and is generated by the points $V_{0,1,1},$  $V_{0,1,1},$ $V_{1,0,1},$ $V_{1,1,1}$. The lattice $(\mathcal{L}(O_5),\langle\cdot,\cdot\rangle)$ is isometric to the $D_4$ lattice.
\end{proposition}
\begin{proof}
	We omit the proof since it is analogous to that of Proposition \ref{prop:Orbit_4_structure}.
\end{proof}

\subsubsection{Orbits of size 36}
The orbit $O_6$ is of size $36$ and is defined by

\begin{align*}
	a_4^2&=- 3B,\  b^6=-108AB,\ c^3=-4B,\\
	a= -\frac{1}{6B}b^2c^2, a_1&=-\frac{1}{6B}a_4b^3,\ a_2=-\frac{1}{3B}a_4b^2c,\ a_3=-\frac{1}{2B}a_4bc^2.
\end{align*}
Let $W_{j,k,m,n}$ denote a point
\[\begin{split}
	\left(
	-2 \sqrt[3]{A} \left(\zeta _3+1\right) \zeta _3^{2 j+2 k}t^2+\left(\zeta _3+2\right) (-1)^m \sqrt[6]{4 A B} \zeta _3^k t+\sqrt[3]{-4 B} \zeta _3^j,\right.\\
	\left.3 \sqrt{A} (-1)^m t^3+2 \sqrt[3]{2 A} \sqrt[6]{B} \left(\zeta _3-1\right)  \zeta _3^{j+2 k}t^2\right.\\ \left.-3 \sqrt[6]{A} \sqrt[3]{4 B} (-1)^m t \zeta _3^{2 j+k+1}+\sqrt{-3B} \right).
\end{split}\]

\begin{proposition}\label{prop:Orbit_6_structure}
	Let $A, B$ be non-zero elements from $\mathbb{Q}$. It follows that 
	\[\mathcal{P}(O_6)=\{W_{j,k,m,n}: j,k \in \{0,1,2\}, m,n\in\{0,1\} \}.\]
	The linear span $\mathcal{L}(O_6)$ has rank $6$ and is generated by the points $W_{0,0,0,0},$ $W_{0,1,0,0},$ $W_{0,0,1,0}$, $W_{0,1,1,0},$ $W_{1,0,0,0},$ $W_{1,1,0,0}$. The lattice $(\mathcal{L}(O_6),\langle\cdot,\cdot\rangle)$ is isometric to the $E_6$ lattice.
\end{proposition}
\begin{proof}
	We have the following relations among the points in the set  $\mathcal{P}(O_6)$:
	\begin{align*}
		W_{j,k,m,0}+W_{j,k,m,1}&=0\quad\textrm{for any $j,k,m$,}\\
		\sum_{j=0}^{2}W_{j,j+s,m,n}&=0\quad\textrm{for any $m,n,s$,}\\
		\sum_{k=0}^{2}W_{j,k,m,n}&=0\quad\textrm{for any $j,m,n$,}\\
		W_{1,1,1,0}&=W_{0,1,1,0}+W_{1,1,0,0}-W_{0,1,0,0},\\
		W_{1,0,1,0}&=W_{0,0,1,0}+W_{1,0,0,0}-W_{0,0,0,0}.
	\end{align*}
	From those relations it follows that there are at most $6$ linearly independent points in $\mathcal{P}(O_6)$ and we check by a direct computation that they span an $E_6$ type lattice.
\end{proof}

The orbit $O_7$ has size $18$ and is defined by
\begin{align*}
	b^6&= -108AB,\ c^3= 8B,\ a_4^2= 9B\\
	a= \frac{1}{12B}b^2c^2,\ a_1&= \frac{1}{18B}a_4b^3,\ a_2= \frac{1}{6B}a_4b^2c,\ a_3= \frac{1}{6B}a_4bc^2
\end{align*}
Let $X_{j,k,m,n}$ denote a point where $j,k\in\{0,1,2\}$ and $m,n\in\{0,1\}$. The $x$-coordinate of $X_{j,k,m,n}$ is
\[\left(
\zeta _3^k \left(t^2 \left(\sqrt[3]{4 A} \left(\zeta _3+1\right) \zeta _3^{2 j}\right)+ \left(\left(\zeta _3+2\right) (-1)^n \sqrt[6]{4 A B}\right) t+2 \sqrt[3]{B} \zeta _3^j\right) \right.\]
and the $y$-coordinate is
\[
\begin{split}
	(-1)^m \left( \left(\sqrt{A} \left(2 \zeta _3+1\right) (-1)^n\right)t^3\right.
	+ \left(3 \sqrt[3]{4 A} \sqrt[6]{B} \left(\zeta _3+1\right) \zeta _3^j\right) t^2 \\
	+\left(2 \sqrt[6]{A} \sqrt[3]{2 B} \left(\zeta _3+2\right) (-1)^n \zeta _3^{2 j}\right) t
	\left.+3 \sqrt{B}\right).
\end{split}
\]
\begin{proposition}\label{prop:Orbit_7_structure}
	Let $A, B$ be non-zero elements from $\mathbb{Q}$. It follows that 
	\[\mathcal{P}(O_7)=\{X_{j,k,m,n}: j,k \in \{0,1,2\}, m,n\in\{0,1\} \}.\]
	The linear span $\mathcal{L}(O_7)$ has rank $6$ and is generated by the points $$X_{0,0,0,0}, X_{0,1,0,0},X_{0,0,0,1}, X_{0,1,0,1},X_{1,0,0,0},X_{1,1,0,0}.$$ The lattice $(\mathcal{L}(O_7),\langle\cdot,\cdot\rangle)$ is isometric to the $E_6$ lattice.
\end{proposition}
\begin{proof}
	We omit the proof since it is analogous to that of Proposition \ref{prop:Orbit_6_structure}.
\end{proof}

\subsubsection{Orbit of size 108}
The orbit $O_8$ is the longest one, of size $108$, with the following defining polynomial equations
\begin{align*}
	0&=64 B^3+48 B^2 c^3+228 B c^6+c^9,\\
	b^6&=\frac{4 A \left(7 B^2+5 B c^3+25 c^6\right)}{3 B},\\
	a_4^2&=B+c^3
\end{align*}

and 

\begin{align*}
	a&=-\frac{b^2 c^2 \left(-11504 B^2+1544 B c^3+7 c^6\right)}{5184 B^3},\\
	a_1&=\frac{a_4 b^3 \left(2728 B^2+4340 B c^3+19 c^6\right)}{1296 B^3},\\
	a_2&=\frac{a_4 b^2 \left(4160 B^2 c+2524 B c^4+11 c^7\right)}{2592 B^3},\\
	a_3&=-\frac{a_4 b c^2 \left(-179 B^2+227 B c^3+c^6\right)}{162 B^3}
\end{align*}
Let $p(x)=\left(x^3+6 x^2+4\right) \left(x^6-6 x^5+36 x^4+8 x^3-24 x^2+16\right)$ be a polynomial. We observe that for $c=\alpha B^{1/3}$ the equation $64 B^3+48 B^2 c^3+228 B c^6+c^9=0$ is equivalent to $p(\alpha)=0$. Let $c_1(o)$ denote the $o$-th root of the polynomial $p(x)$ defined by the formula
\[c_{1}(o)=-2^{2/3}\left(\sqrt[3]{2} \zeta _3^{2 \left(\left\lfloor \frac{o-1}{3}\right\rfloor +o-1\right)}+\zeta _3^{\left\lfloor \frac{o-1}{3}\right\rfloor }+2^{2/3} \zeta _3^{o-1}\right)\]
for $o \in \{1,\ldots,9\}$.

Let $Y_{o,j,m,n}$ denote a point
\[Y_{o,j,m,n} =(a t^2+bt+c,a_1 t^3+a_2 t^2+a_3 t+a_4) \]
for $m,n\in\{-1,1\}$, $j\in\{0,1,2\}$ and $o=1,\ldots,9$ and
where $c=c(o)=c_1(o)B^{1/3}$, $a_4=a_4(n,o)=(-1)^{n}B^{1/2}(1+c_1(o))^{1/2}$ and $b=b(j,m,o)=(-1)^m(AB)^{1/6}(4/3 (7 + 5 c_1(o)^3 + 25 c_1(o)^6))^{1/6}\zeta_3^j$.
It is easy to check that the suitable roots of degree $2$ and $6$ in the expressions above belong to the field $\mathbb{Q}(\zeta_3,2^{1/3})$.
\begin{proposition}\label{prop:Orbit_8_structure}
	Let $A, B$ be non-zero elements from $\mathbb{Q}$. It follows that 
	\[\mathcal{P}(O_8)=\{Y_{o,j,m,n}: j \in \{0,1,2\}, m,n\in\{0,1\},o\in \{1,\ldots,9\} \}.\]
	The linear span $\mathcal{L}(O_8)$ has rank $8$ and is generated by the points
	$Y_{1, 0, 0, 0},$ $Y_{1, 0, 1, 0},$ $Y_{1, 1, 0, 0}, Y_{1, 1, 1, 0},$ $Y_{1, 2, 0, 0}, Y_{3,0,0,0},Y_{3, 0, 1, 0}, Y_{4, 0, 0, 0}$. The lattice $\mathcal{L}(O_8)$ is isometric to the $E_8$ lattice.
\end{proposition}
\begin{proof}
	We have the following equalities
	$$Y_{o,j,m,0}+Y_{o,j,m,1}=0,$$
	$$Y_{1,j,m,0}+Y_{3,j,m,0}+Y_{2,j,m+1,0}=0,$$
	$$Y_{4,j,m,0}+Y_{5,j,m,0}+Y_{6,j,m+1,0}=0,$$
	$$Y_{8,j,m,0}+Y_{9,j,m,0}+Y_{7,j,m+1,0}=0,$$	
	for any choice of $o,j,m$. We also have the following relation
	$$\sum_{j,m}Y_{o,j,m,0}=0$$ for any $o\in \{1,3,4,5,8,9\}$.
	Relations above allow us to reduce the spanning set for $\mathcal{L}(O_{8})$ to $30 $ points among which we find pairs which add up to the same points. 
	This generates another $21$ relations and we compute the height pairing matrix for the remaining $9$ points. 
	The height matrix has rank $8$ and its kernel provides the final relation. 
	Finally we check that the lattice of rank $8$ that we have obtained has discriminant $1$, thus it must be isomorphic to $E_{8}$ since the whole Mordell-Weil lattice $E_{A,B}(\overline{\mathbb{Q}}(t))$ is isomorphic to the latter by \cite{Oguiso_Shioda}.
\end{proof}
\begin{remark}
	We computed the relation between the roots of the Mordell-Weil lattice in terms of orbits.
	Let $R(G)$ denote the subset of elements in $G\subset E_{A,B}(\overline{\mathbb{Q}}(t))$ of height 2. The following equalities hold
	\begin{align*}
		\RLL[1]&=\PO[1]\\
		\RLL[2]&=\PO[2]\\
		\RLL[3]&=\PO[3]\\
		\RLL[4]&=\PO[1]\cup\PO[4]\\
		\RLL[5]&=\PO[2]\cup\PO[5]\\
		\RLL[6]&=\PO[2]\cup\PO[3]\cup\PO[5]\cup\PO[6]\\
		\RLL[7]&=\PO[1]\cup\PO[3]\cup\PO[4]\cup\PO[7]\\
		\RLL[8]&=\bigcup_{i=1}^{8}\PO[i]\\
	\end{align*}
\end{remark}

\section{Proof of the main theorem}\label{sec:proof_main_theorem}
\subsection{Structure of the Galois modules}\label{sec:Galois_module_structure}
Let $A, B$ be non-zero rational numbers. In this section we compute in detail the Galois action of the group $G=\Gal(\overline{\mathbb{Q}}/\mathbb{Q})$ on the module $M=E_{A,B}(\overline{\mathbb{Q}}(t))$ and the $\mathbb{Q}$-vector space $V=M\otimes\mathbb{Q}$. It follows from the discussion in Section \ref{sec:preliminaries_on_MW_groups} that $V$ is a vector space of dimension $8$ over $\mathbb{Q}$. We define four subspaces $V_i\subset V$, $i=1,2,3,4$ by their $\mathbb{Q}$-generators

\begin{align*}
	V_{1} &=\myspan_{\mathbb{Q}}\langle P,P^{\sigma_{K}}\rangle,\\
	V_{2} &=\myspan_{\mathbb{Q}}\langle Q, Q^{\sigma_{K'}}\rangle,\\
	V_{3} &=\myspan_{\mathbb{Q}}\langle R+S,R^{\sigma_{L}}+S^{\sigma_{L}}\rangle,\\
	V_{4} &=\myspan_{\mathbb{Q}}\langle R-S,R^{\sigma_{L}}-S^{\sigma_{L}}\rangle.\\
\end{align*}

\begin{proposition}
	Let $A, B$ be non-zero rational numbers. The subspaces $V_i$ are $\mathbb{Q}[G]$-sub-modules  of $V=E_{A,B}(\overline{\mathbb{Q}}(t))\otimes\mathbb{Q}$. The $\mathbb{Q}[G]$-module $V$ is a direct sum $V_{1}\oplus V_{2}\oplus V_{3}\oplus V_{4}$ as $\mathbb{Q}[G]$-modules.
	and each subspace $V_{i}$ is of dimension $2$ over $\mathbb{Q}$.
\end{proposition}

\begin{proof}
	Proposition \ref{prop:Point_P_properties} implies that the submodule $V_1$ is stable under the action of the group $G$. Proposition \ref{prop:Point_Q_properties} implies that the submodule $V_{2}$ is $G$-stable. Proposition \ref{prop:Point_RS_properties} implies that $V_3$ and $V_4$ are $G$-stable.
	
	To prove that each $V_i$ has dimension $2$ over $\mathbb{Q}$ we observe that the chosen spanning sets have the Gram matrix $M$ with respect to the height pairing  (defined in Proposition \ref{prop:Point_P_properties}). Hence, the assumption that there exists two non-zero $x,y\in\mathbb{Q}$ such that $x g_1+y g_2=0$ for $g_1,g_2$ - spanning elements of $V_i$, would imply that $M\cdot(x,y)^{T} = (0,0)^{T}$ and since $\det M=3$, $(x,y) = (0,0)$.
	
	Next, we show that $V$ is a direct sum of $V_i$ submodules.
	Equivalently, we show that the intersections $V_{i}\cap V_{j}$ are zero. The equality $V_{i}\cap V_{j}=0$ for $i=1,2$ and $j\neq i$ follows directly from the structure of the Gram matrix for the set $\{P,P^{\sigma_{K}}, Q, Q^{\sigma_{K'}}, R, R^{\sigma_{L},} S, S^{\sigma_{L}}\}$ described in Proposition \ref{prop:Finite_index_sublattice}. To verify $V_{3}\cap V_{4}=0$ we check that $\langle R+S,R-S\rangle =0$ and $\langle R+S, R^{\sigma_{L}}-S^{\sigma_{L}}\rangle = 0$ and we do a similar calculation for $R-S$. 
\end{proof}

Since the elliptic surface $\mathcal{E}_{A,B}$ has no torsion sections (Proposition \ref{prop:pt_structure_on_ell_surf}) it follows that $V^{G} = E_{A,B}(\mathbb{Q}(t))\otimes\mathbb{Q}$ and also $E_{A,B}(\mathbb{Q}(t))=E_{A,B}(\overline{\mathbb{Q}}(t))\cap V^{G}$. We denote by $V_{i}^{G}$ the $G$-invariants of a submodule $V_i$.
\begin{proposition}\label{prop:suff_condition_Vi}
	Let $A,B$ be non-zero rational numbers and $E=E_{A,B}$. The following conditions hold:
	\begin{itemize}
		\item If $V_{1}^{G}\neq \{0\}$, then $B$ is a cube.
		\item If $V_{2}^{G}\neq \{0\}$, then $A$ is a cube.
		\item If $V_{3}^{G}\neq \{0\}$, then $4AB$ is a cube.
		\item If $V_{4}^{G}\neq \{0\}$, then $4AB$ is a cube.
	\end{itemize}
\end{proposition}
\begin{proof}
	Let $X$ denote $B$, $A$ or $4AB$. If $X$ is not a cube in $\mathbb{Q}$, then there exists an automorphism in $\sigma\in\Gal(\overline{\mathbb{Q}}/\mathbb{Q})$ which restricts to respectively $\sigma_{K}$, $\sigma_{K'}$ or $\sigma_{L}$. In the $\mathbb{Q}$-basis of $V_{i}$ the matrix of $\sigma$ is
	\[M:=\left(\begin{array}{cc}
		0 & -1 \\
		1 & -1 \\
		
	\end{array}\right)\]
	which has characteristic polynomial $1+x+x^2$, hence there are no fixed vectors in the representation attached to $V_{i}$, so $V_{i}\cap E(\mathbb{Q}(t))=\{0\}$.
\end{proof}

Let $r_i$ denote the dimension of the $\mathbb{Q}$-space $V_{i}^{G}$. Note that $r_i$ is also the $\mathbb{Z}$-rank of the module $E_{A,B}(\overline{\mathbb{Q}}(t))\cap V_{i}^{G}$. Let $x=x'\cdot\square$ denote a number $x$ which is a product of $x'$ with a square in~$\mathbb{Q}^{\times}$.
\begin{proposition}[Analyse representation $V_{1}$]\label{prop:V1_special}
	Assume $B$ is a cube. Then one of the following conditions is true.
	\begin{itemize}
		\item[(i)] If $A\neq-3\times \square$, then $r_1=0$.
		\item[(ii)] If $A=\square$, then $r_1=1$ and $V_{1}^{G}=\langle P\rangle.$
		\item[(iii)] If $A=-3\times \square$, then $r_{1}=1$ and $V_{1}^{G}=\langle P+2P^{\sigma_{K}}\rangle.$
	\end{itemize}
\end{proposition}
\begin{proof}
	(i): the assumptions imply that there exists an automorphism $\theta$ in $\Gal(\mathbb{Q}(\sqrt{A},\zeta_{3})/\mathbb{Q})$ which satisfies $\theta(\sqrt{A})=\sqrt{A}$ and $\theta(\zeta_{3})=-\zeta_{3}$. Hence, in the given basis of $V_{1}$ it has matrix $\matr{1}{-1}{0}{-1}$ and an eigenvector $P\otimes 1$ with eigenvalue $1$. This vector is not fixed under the automorphism that sends $\sqrt{A}$ to $-\sqrt{A}$ and fixes $\zeta_{3}$. Hence $r_1=0$.
	
	(ii): In this case the automorphism $\theta$ again fixes only a $1$-dimensional subspace spanned by $P\otimes 1$ which belongs to $E(\mathbb{Q}(t))\otimes\mathbb{Q}$, hence $r_1=1$.
	
	(iii): In this case the action of $\theta$ in the basis of $V_1$ provides a matrix $\matr{-1}{1}{0}{1}$, hence it fixes the space spanned by the vector $(P+2P^{\sigma_{K}})\otimes 1$ which is rational, hence $r_1=1$.
\end{proof}
Propositions \ref{prop:suff_condition_Vi} and \ref{prop:V1_special} completely characterize the structure of the subpace $V_i^{G}$.

\begin{table}[htb]
	\begin{center}
		\begin{tabular}{  l | l | l |l}
			Point & Formula & Height & Orbit\\
			\hline
			$P$ & $(-\sqrt[3]{B},\sqrt{A} t^3)$ & 2& $O_1$\\
			$Q$ & $(-\sqrt[3]{A}t^2,\sqrt{B})$ & 2& $O_2$\\
			$R$ & $(\frac{2\sqrt{A}\sqrt{B}}{s}t,\sqrt{A}t^3+\sqrt{B})$ & 2& $O_3$\\
			$S$ & $(-\frac{2\sqrt{A}\sqrt{B}}{s}t,\sqrt{A}t^3-\sqrt{B})$ & 2& $O_3$\\
		\end{tabular}
		\caption{Points on the elliptic surface $\E:y^2=x^3+At^6+B$.}
		\label{tab:generators}
	\end{center}
\end{table}
\begin{proposition}[(Analyse representation $V_{2}$)]
	Assume $A$ is a cube. 
	\begin{itemize}
		\item[(i)] If $B\neq-3\times \square$, then $r_2=0$.
		\item[(i)] If $B= \square$, then $r_2=1$ and $V_{2}^{G}=\langle Q \rangle$.
		\item[(ii)] If $B=-3\times \square$, then $r_2=1$ and $V_{2}^{G}=\langle Q+2Q^{\sigma_{K'}} \rangle$.
	\end{itemize}
\end{proposition}
\begin{proof}
	We omit the proof since it is analogous to that of Proposition \ref{prop:V1_special}.
\end{proof}

\begin{proposition}[(Analyse representation $V_{3}$)]
	Assume $4AB$ is a cube. 
	\begin{itemize}
		\item[(i)] If $A\neq-3\times \square$, then $r_3=0$.
		\item[(ii)] If $A=\square$, then $r_3=1$ and $V_{1}^{G}=\langle R+S\rangle.$
		\item[(iii)] If $A=-3\times \square$, then $r_{3}=1$ and $$V_{3}^{G}=\langle R+S+2(R^{\sigma_{L}}+S^{\sigma_{L}})\rangle.$$
	\end{itemize}
\end{proposition}
\begin{proof}
	We omit the proof since it is analogous to that of Proposition \ref{prop:V1_special}.
\end{proof}

\begin{proposition}[(Analyse representation $V_{4}$)]
	Assume $4AB$ is a cube. 
	\begin{itemize}
		\item[(i)] If $B\neq-3\times \square$, then $r_4=0$.
		\item[(ii)] If $B=\square$, then $r_4=1$ and $V_{4}^{G}=\langle R-S\rangle.$
		\item[(iii)] If $B=-3\times \square$, then $r_{4}=1$ and $V_{4}^{G}=\langle R-S+2(R^{\sigma_{L}}-S^{\sigma_{L}})\rangle.$
	\end{itemize}
\end{proposition}
\begin{proof}
	We omit the proof since it is analogous to that of Proposition \ref{prop:V1_special}.
\end{proof}

\subsection{Decision diagrams}\label{sec:Decision_diagram}
The results of Section \ref{sec:Galois_module_structure} are sufficient to conclude the value of the rank $r_{\mathcal{E}}$ of $E_{A,B}(\mathbb{Q}(t))$ based on the execution of the procedure which on input takes a pair of non-zero rational numbers $A,B$ and prints on the output the value $r=r_{\mathcal{E}}$ (indicated by a red rectangular box in the diagrams).

The starting point for the procedure is the diagram on Figure \ref{fig:Diagram_0}. Each diamond box is a query with possible yes or no answer. Blue circles with numbers $1,2,3$ denote the starting point of a subroutine explained on Figures \ref{fig:Subroutine_1}, \ref{fig:Subroutine_2} and \ref{fig:Subroutine_3}.

\begin{figure}
	\centering
	\begin{tikzpicture}[sibling distance=1em,level distance=2.5cm]
		\Tree [.\node[start]{START};
		[.\node[test]{$[\mathbb{Q}(\sqrt{A},\zeta_3):\mathbb{Q}]=4$}; 
		\edge node[ auto=right]{YES};
		[.\node[test]{$[\mathbb{Q}(\sqrt{B},\zeta_3):\mathbb{Q}]=4$};
		\edge node[auto=right]{YES};  
		\node[proc]{$r=0$}; 
		\edge node[auto=left]{NO};  
		\node[nmark]{1};
		]
		\edge node[ auto=left]{NO}; 
		[.\node[test]{$[\mathbb{Q}(\sqrt{B},\zeta_3):\mathbb{Q}]=4$}; 
		\edge node[auto=right]{YES};  
		\node[nmark]{2};
		\edge node[auto=left]{NO};  
		\node[nmark]{3};
		] 
		]
		]
		]
		
	\end{tikzpicture}
	\caption{Initial point of the procedure}\label{fig:Diagram_0}
\end{figure}
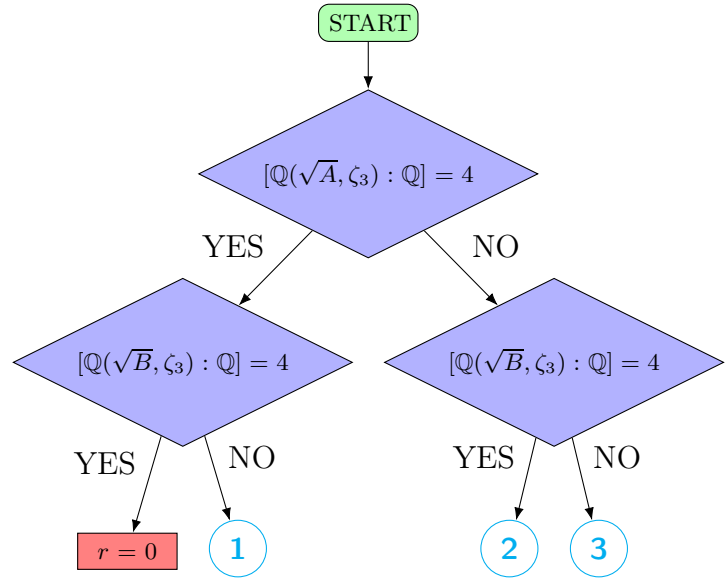

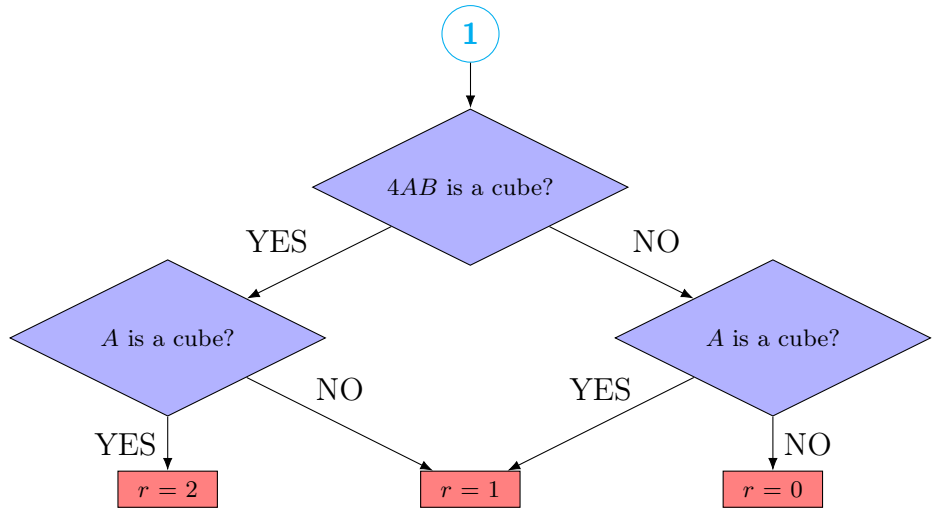
\begin{figure}
	\centering
	\begin{tikzpicture}[sibling distance=3em]
		\Tree 
		[.\node[nmark]{1};
		[.\node[test]{$4AB$ is a cube?}; 
		\edge node[ auto=right]{YES};
		[.\node(Ac)[test]{$A$ is a cube?};
		\edge node[auto=right]{YES};  
		\node[proc]{$r=2$}; 
		] 
		\edge[draw=none];
		[.{} 
		\edge[draw=none];	
		[.\node(a)[proc]{$r=1$}; ]
		]		
		\edge node[ auto=left]{NO}; 
		[.\node(Ac2)[test]{$A$ is a cube?}; 
		\edge node[auto=left]{NO};  
		\node[proc]{$r=0$}; 
		] 
		]
		]
		\draw[->,-Latex]  (Ac2)--(a) node [midway, above=1ex] {YES}; 
		\draw[->,-Latex]  (Ac)--(a) node [midway, above=1ex] {NO}; 
	\end{tikzpicture}
	\caption{Subroutine 1}\label{fig:Subroutine_1}
\end{figure}

\begin{figure}
	\centering
	\begin{tikzpicture}[sibling distance=3em]
		\Tree 
		[.\node[nmark]{2};
		[.\node[test]{$4AB$ is a cube?}; 
		\edge node[ auto=right]{YES};
		[.\node(Ac)[test]{$B$ is a cube?};
		\edge node[auto=right]{YES};  
		\node[proc]{$r=2$}; 
		] 
		\edge[draw=none];
		[.{} 
		\edge[draw=none];	
		[.\node(a)[proc]{$r=1$}; ]
		]		
		\edge node[ auto=left]{NO}; 
		[.\node(Ac2)[test]{$B$ is a cube?}; 
		\edge node[auto=left]{NO};  
		\node[proc]{$r=0$}; 
		] 
		]
		]
		\draw[->,-Latex]  (Ac2)--(a) node [midway, above=1ex] {YES}; 
		\draw[->,-Latex]  (Ac)--(a) node [midway, above=1ex] {NO}; 
	\end{tikzpicture}
	\caption{Subroutine 2}\label{fig:Subroutine_2}
\end{figure}
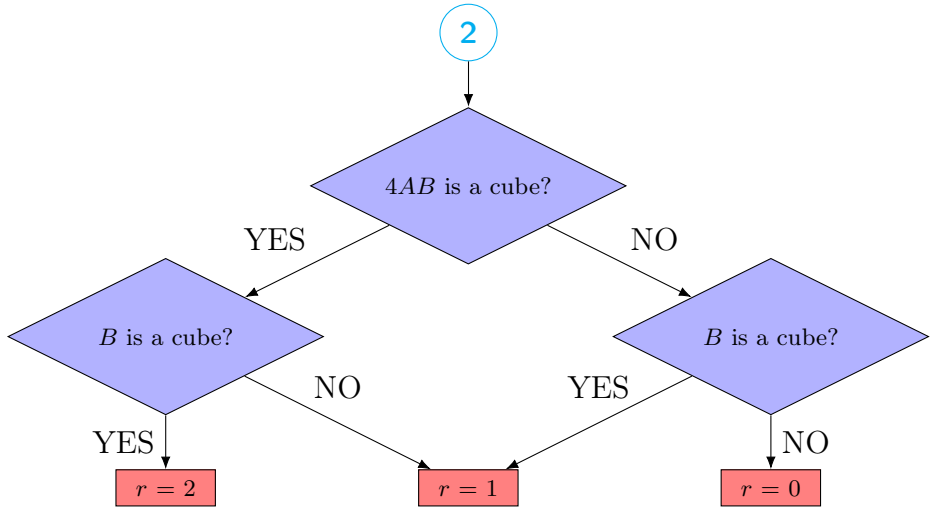

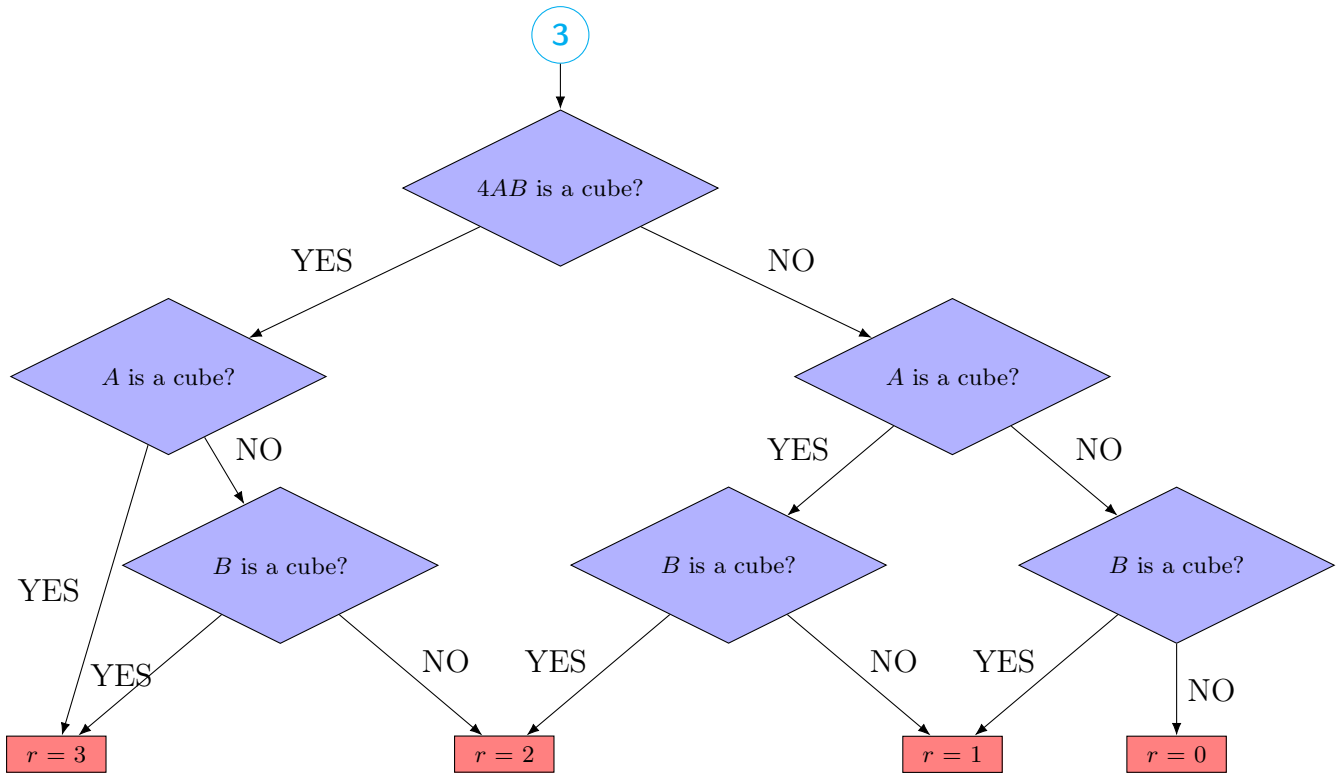
\begin{figure}
	\centering
	\begin{tikzpicture}[sibling distance=0.5em,level distance=2.5cm]
		\Tree [.\node[nmark]{3};
		[.\node[test]{$4AB$ is a cube?}; 
		\edge node[ auto=right]{YES};
		[.\node(Ac1)[test]{$A$ is a cube?};
		\edge[draw=none];
		[.{} 
		\edge[draw=none];
		[.\node(r3)[proc]{$r=3$}; ]
		]
		\edge node[auto=left]{NO};
		[.\node(Bc1)[test]{$B$ is a cube?}; 
		]
		]  
		\edge[draw=none];
		[.{}
		\edge[draw=none];
		[.{}
		\edge[draw=none]; 
		[.\node(r2)[proc]{$r=2$}; ]
		] 
		]
		\edge node[ auto=left]{NO}; 
		[.\node[test]{$A$ is a cube?}; 
		\edge node[auto=right]{YES};  
		[.\node(Bc2)[test]{$B$ is a cube?}; 
		]
		\edge[draw=none];
		[.{}
		\edge[draw=none]; 
		[.\node(r1)[proc]{$r=1$}; ] 
		]
		\edge node[auto=left]{NO};
		[.\node(Bc3)[test]{$B$ is a cube?}; 
		\edge node[auto=left]{NO};
		[.\node[proc]{$r=0$};] 
		] 
		] 
		]
		]
		\draw[->,-Latex]  (Bc2)--(r1) node [midway, above=1ex,right] {NO}; 
		\draw[->,-Latex]  (Bc3)--(r1) node [midway, above=1ex,left] {YES}; 
		\draw[->,-Latex]  (Bc2)--(r2) node [midway, above=1ex,left] {YES}; 
		\draw[->,-Latex]  (Bc1)--(r2) node [midway, above=1ex,right] {NO};
		\draw[->,-Latex]  (Ac1)--(r3) node [midway, left=1ex] {YES}; 
		\draw[->,-Latex]  (Bc1)--(r3) node [midway, above=1ex,left=-1ex] {YES}; 	
	\end{tikzpicture}
	\caption{Subroutine 3}\label{fig:Subroutine_3}
\end{figure}

\subsection{Rational basis}\label{sec:rational_basis}
We compute below the set of generators of the group $E_{A,B}(\mathbb{Q}(t))$ for each choice of non-zero rational numbers $A,B$. We follow the structure of the diagrams from Section \ref{sec:Decision_diagram}. We denote each paragraph with the label that corresponds to yes/no decisions made in the decision diagram (abbreviated as Y/N respectively) in order to reach the specific basis.

We use the following facts for the proofs below:
\begin{itemize}
	\item[$(\dagger)$] Height function is quadratic and the minimal height of the non-zero point in $E_{A,B}(\overline{\mathbb{Q}}(t))$ is $2$. There are no points of height $3$ in the group $E_{A,B}(\overline{\mathbb{Q}}(t))$. In particular there are no such points in the group $E_{A,B}(\mathbb{Q}(t))$.
	\item The span of points $\{e_{i}\}$ is not $m$-saturated in $E_{A,B}(\overline{\mathbb{Q}}(t))$ if and only if one can find a linear combination $\sum_{i} m_i e_i$ with $m_i\in\{0,\ldots,m-1\}$ which is $m$-divisible and non-zero.
	\item Let $G$ denote the Gram matrix of the span $S$ of points $\{e_{i}\}_{i=1}^{m}$ for some positive integer $m$. If $\det G\neq 0$, then it is an integer. If $S$ is not $m$-saturated, then $m^2$ divides $\det G$.
\end{itemize}

\begin{table}[htb]
	\begin{center}
		\begin{tabular}{  l | l | l }
			\hline
			Point & Formula & Height\\
			\hline
			$R+S$ & $\left(\frac{1}{t^2}\frac{B}{\sqrt[3]{4AB}},\frac{1}{t^3}\left(-\sqrt{A}t^6-\frac{B}{2\sqrt{A}}\right) \right)$ & 4 \\
			$R-S$ & $\left(t^4\frac{A}{\sqrt[3]{4AB}},-\frac{A}{2\sqrt{B}}t^6-\sqrt{B}\right)$ & 4\\
			$P+2P^\sigma$ & $(\frac{4A}{-3\sqrt[3]{B}}t^6-\sqrt[3]{B},\frac{8A\sqrt{-3A}}{-3^2B}t^9-3\sqrt{-3A}t^3)$ & 6\\
			$Q+2Q^\sigma$ & $\left(\frac{4B}{-3\sqrt[3]{A}t^4}-\sqrt[3]{A}t^2, \frac{8B\sqrt{-3B}}{-3^2At^6}-3\sqrt{-3B}\right)$ & 6\\
			$R+S+2(R^{\sigma}+S^{\sigma})$ & $\left(-\frac{16 A^2 t^{12}+16 A B t^6+B^2}{3 B  \sqrt[3]{4 A B}t^2},-\frac{\left(2 A t^6+B\right) \left(-32 A^2 t^{12}-32 A B t^6+B^2\right)}{6 B^2 \sqrt{-3 A} t^3}\right)$& 12 \\
			$R-S+2(R^{\sigma}-S^{\sigma})$ & $\left(-\frac{A^2 t^{12}+16 A B t^6+16 B^2}{3 A  \sqrt[3]{4 A B} t^8},-\frac{\left(A t^6+2 B\right) \left(A^2 t^{12}-32 A B t^6-32 B^2\right)}{6 A^2 \sqrt{-3 B} t^{12}}\right)$ & 12 \\
			\hline
		\end{tabular}
		\caption{\label{tab:further_points}Points on the elliptic curve $E_{A,B}:y^2=x^3+At^6+B$.
		}
	\end{center}
\end{table}

\newcommand{\Fext}[2]{[\mathbb{Q}(\sqrt{#1},\zeta_3):\mathbb{Q}]=#2}

\newcommand{\IsCube}[2]{#1 #2 \eta^3}
\renewcommand{\and}{\ \wedge\ }
\newcommand{\condbox}[2]{\fbox{\parbox{#1}{\centering #2}}}
\newcommand{\ratbasis}[4]{\begin{itemize}
		\item $r_{1} = #1$, $E_{A,B}(\overline{\mathbb{Q}}(t))\cap V_1^{G}$ basis: \ifnum#1=1 {$P$ or $P+2P^{\sigma}$} \else {empty}\fi
		\item $r_{2} = #2$, $E_{A,B}(\overline{\mathbb{Q}}(t))\cap V_2^{G}$ basis: \ifnum#2=1 {$Q$ or $Q+2Q^{\sigma}$} \else {empty}\fi
		\item $r_{3} = #3$, $E_{A,B}(\overline{\mathbb{Q}}(t))\cap V_3^{G}$ basis: \ifnum#3=1 {$R+S$ or $R+S+2(R+S)^{\sigma}$} \else {empty}\fi
		\item $r_{4} = #4$, $E_{A,B}(\overline{\mathbb{Q}}(t))\cap V_4^{G}$ basis: \ifnum#4=1 {$R-S$ or $R-S+2(R-S)^{\sigma}$} \else {empty}\fi
	\end{itemize}
}
\newcommand{\mydiag}[2]{
	\left(\begin{array}{cc} %
		#1 & 0\\ %
		0 & #2 %
	\end{array}\right)
}

\newcommand{\threediag}[3]{
	\left(\begin{array}{ccc} %
		#1 & 0 & 0\\ %
		0 & #2 & 0\\%
		0 & 0 & #3 %
	\end{array}\right)
}

\subsubsection{When $r_\E=1$}

In the following cases, the rank of $E_{A,B}(\mathbb{Q}(t))$ is $1$ and the union of the bases of the submodules $E_{A,B}(\overline{\mathbb{Q}}(t))\cap V_i^{G}$ gives the full basis of the Mordell-Weil lattice simply by the property $(\dagger)$. We give two choices for the basis, but in each case it can easily be determined: it is $P$ or $R+S$ (resp. $Q$ or $R-S$) if $A$ (resp. $B$) is a square, and $P+2P^\sigma$ or $R+S+2(R+S)^\sigma$ (resp. $Q+2Q^\sigma$ or $R-S+2(R-S)^\sigma$) if $A$ (resp. $B$) is $-3$ times a square.

\subsubsection*{YNYN}
\condbox{0.8\textwidth}{$\Fext{A}{4}\and\Fext{B}{2}\and\IsCube{4AB}{=}\and\IsCube{A}{\neq}$}
\ratbasis{0}{0}{0}{1}

\subsubsection*{YNNY}
\condbox{0.8\textwidth}{$\Fext{A}{4}\and\Fext{B}{2}\and\IsCube{4AB}{\neq}\and\IsCube{A}{=}$}
\ratbasis{0}{1}{0}{0}

\subsubsection*{NYNY}
\condbox{0.8\textwidth}{$\Fext{A}{2}\and\Fext{B}{4}\and\IsCube{4AB}{\neq}\and\IsCube{B}{=}$}
\ratbasis{1}{0}{0}{0}

\subsubsection*{NYYN}
\condbox{0.8\textwidth}{$\Fext{A}{2}\and\Fext{B}{4}\and\IsCube{4AB}{=}\and\IsCube{B}{\neq}$}
\ratbasis{0}{0}{1}{0}

\subsubsection*{NNNYN}
\condbox{0.8\textwidth}{$\Fext{A}{2}\and\Fext{B}{2}\and\IsCube{4AB}{\neq}\and\IsCube{A}{=}\and\IsCube{B}{\neq}$}
\ratbasis{0}{1}{0}{0}

\subsubsection*{NNNNY}
\condbox{0.8\textwidth}{$\Fext{A}{2}\and\Fext{B}{2}\and\IsCube{4AB}{\neq}\and\IsCube{A}{\neq}\and\IsCube{B}{=}$}
\ratbasis{1}{0}{0}{0}

\subsubsection{When $r_\E=2$}
In each case we have to verify whether the union of the bases of the submodules $E_{A,B}(\overline{\mathbb{Q}}(t))\cap V_{i}^{G}$ are saturated in $E_{A,B}(\mathbb{Q}(t))$. 
We compute for every such choice a basis of points with least heights and a lattice type corresponding to that basis. We denote by $\langle n\rangle$ a rank $1$ lattice with the height pairing $b(\cdot,\cdot)$ such that the generator $e$ of the lattice satisfies $b(e,e)=n$.  We denote by $A_2(2)$ a lattice of rank $2$ such that the Gram matrix of the height pairing has the form 
\[\left(\begin{array}{cc}
	4 & -2\\
	-2 & 4
\end{array}\right).\]
We denote by $\Lambda_1\oplus\Lambda_2$ the orthogonal sum of lattices $\Lambda_1$, $\Lambda_2$. We denote by $\diag (a_1,\ldots,a_n)$ an $n$ by $n$ diagonal matrix with entries $a_i$ on the diagonal, ordered from the top-left to the bottom-right.
\subsubsection*{YNYY}
\condbox{0.8\textwidth}{$\Fext{A}{4}\and\Fext{B}{2}\and\IsCube{4AB}{=}\and\IsCube{A}{=}$}
\ratbasis{0}{1}{0}{1}

--- $B=\square$: Points $Q$ and $R-S$ form a basis of the full Mordell-Weil group, since the lattice has discriminant $8$ and each non-zero point $\alpha Q+\beta (R-S)$ for $\alpha,\beta\in\{0,1\}$ has height in the set $\{2,4,6\}$.

\medskip
\noindent
\textbf{Minimal height basis:} $Q,R-S$

\noindent
\textbf{Lattice type:} $\langle 2\rangle\oplus \langle 4\rangle$

\medskip
--- $B=-3\times\square$: Point $Q+2Q^{\sigma}+(R-S)+2(R-S)^{\sigma}$ 
has height $18$ and is divisible by $3$, namely

$$Q+2Q^{\sigma}+(R-S)+2(R-S)^{\sigma} = -3V_{0,0,1}.$$
The Gram matrix of the pair $Q+2Q^{\sigma}, V_{0,0,1}$ is 

$$\left(\begin{array}{cc}
	6 & -2 \\
	-2 & 2
\end{array}\right)$$

Let $e_1=Q+2Q^{\sigma}$ and  $e_2=V_{0,0,1}$. The height of the non-zero $\alpha_1 e_1+\alpha_2 e_2$ where $\alpha_i\in \{0,1\}$ belongs to the set $\{2, 4, 6 \}$. Hence, no non-zero point of this form  is $2$-divisible. Because the Gram matrix has determinant $8$ it follows that $\{e_1,e_2\}$ is a basis of rational points $E_{A,B}(\mathbb{Q}(t))$.

\medskip
\noindent
\textbf{Minimal height basis:} $V_{0,0,1},Q+2Q^{\sigma}+V_{0,0,1}$

\noindent
\textbf{Lattice type:} $\langle 2\rangle\oplus \langle 4\rangle$

\subsubsection*{NYYY}
\condbox{0.8\textwidth}{$\Fext{A}{2}\and\Fext{B}{4}\and\IsCube{4AB}{=}\and\IsCube{B}{=}$}
\ratbasis{1}{0}{1}{0}

\noindent
--- $A=\square$: the basis of rational points consists of $P$ and $R+S$ since the basis is $2$ saturated and the Gram matrix is the diagonal matrix $\textrm{diag}(2,4)$.

\medskip
\noindent
\textbf{Minimal height basis:} $P,R+S$

\noindent
\textbf{Lattice type:} $\langle 2\rangle\oplus \langle 4\rangle$

\medskip
\noindent
--- $A=-3\times\square$: the points $P+2P^{\sigma}$ and $R+S+2(R+S)^{\sigma}$ form a lattice of index dividing $9$ in the full Mordell-Weil lattice $E_{A,B}(\mathbb{Q}(t))$. We have

$$-3U_{0,0,1} = P+2P^{\sigma}+R+S+2(R+S)^{\sigma}$$

\noindent
and the basis $e_1=P+2P^{\sigma}, e_2=U_{0,0,1}$ has the Gram matrix 

$$\left(\begin{array}{cc}
	6 & -2 \\
	-2 & 2
\end{array}\right).$$
The lattice spanned by $e_1$ and $e_2$ is $2$-saturated.

\medskip
\noindent
\textbf{Minimal height basis:} $U_{0,0,1},P+2P^{\sigma}+U_{0,0,1}$

\noindent
\textbf{Lattice type:} $\langle 2\rangle\oplus \langle 4\rangle$

\subsubsection*{NNYNN}
\condbox{0.8\textwidth}{$\Fext{A}{2}\and\Fext{B}{2}\and\IsCube{4AB}{=}\and\IsCube{A}{\neq}\and\IsCube{B}{\neq}$}
\ratbasis{0}{0}{1}{1}

--- $A=\square$, $B=\square$: the lattice formed by $R+S$ and $R-S$ is not $2$-saturated. Points $R, S$ form the basis of $E_{A,B}(\mathbb{Q}(t))$.

\medskip
\noindent
\textbf{Minimal height basis:} $R,S$

\noindent
\textbf{Lattice type:} $\langle 2\rangle\oplus \langle 2\rangle$

\medskip
--- $A=-3\times\square$, $B=\square$: the basis is $R+R^{\sigma}+S^{\sigma}, R-S$ with the Gram matrix 
$\left(\begin{array}{cc}
	4 & 2 \\
	2 & 4
\end{array}\right)$
of determinant $12$. To show this is a basis it is enough to check that the lattice is $2$-saturated.

\medskip
\noindent
\textbf{Minimal height basis:} $R+R^{\sigma}+S^{\sigma},-(R-S)$

\noindent
\textbf{Lattice type:} $A_2(2)$

\medskip
--- $A=\square$, $B=-3\times\square$:
the basis is $R+R^{\sigma}-S^{\sigma}, R+S$ with the Gram matrix 
$\left(\begin{array}{cc}
	4 & 2 \\
	2 & 4
\end{array}\right).$

\medskip
\noindent
\textbf{Minimal height basis:} $R+R^{\sigma}-S^{\sigma},-(R+S)$

\noindent
\textbf{Lattice type:} $A_2(2)$

\medskip
--- $A=-3\times\square$, $B=-3\times\square$:
sublattice of discriminant $144$ spanned by the points $R-S+2(R-S)^{\sigma},R+S+2(R+S)^{\sigma}$ has the Gram matrix $\diag(12,12)$.
Points $R+2R^{\sigma}, S+2S^{\sigma}$ with the Gram matrix $\diag(6,6)$ are linearly independent and they form a $2$ and $3$ saturated lattice, so they form a basis of the group $E_{A,B}(\mathbb{Q}(t))$.

\medskip
\noindent
\textbf{Minimal height basis:} $R+2R^{\sigma},S+2S^{\sigma}$

\noindent
\textbf{Lattice type:} $\langle 6\rangle \oplus \langle 6\rangle$

\subsubsection*{NNNYY}
\condbox{0.8\textwidth}{$\Fext{A}{2}\and\Fext{B}{2}\and\IsCube{4AB}{\neq}\and\IsCube{A}{=}\and\IsCube{B}{=}$}
\ratbasis{1}{1}{0}{0}

--- $A=\square$, $B=\square$: $P,Q$ form a lattice with the Gram matrix $\diag(2,2)$,
hence the lattice is saturated in the Mordell-Weil lattice $E_{A,B}(\mathbb{Q}(t))$.

\medskip
\noindent
\textbf{Minimal height basis:} $P,Q$

\noindent
\textbf{Lattice type:} $\langle 2\rangle \oplus \langle 2\rangle$

\medskip
--- $A=\square$, $B=-3\times\square$: $P$, $Q+2Q^{\sigma}$ form a lattice with the Gram matrix $\diag(2,6)$,
and this could a priori be not $2$-saturated. That would be equivalent to $P+Q+2Q^{\sigma}$ being $2$-divisible, or equivalently $P+Q$ being two divisible, but this point has height $4$, so cannot be $2$-divisible.

\medskip
\noindent
\textbf{Minimal height basis:} $P,Q+2Q^{\sigma}$

\noindent
\textbf{Lattice type:} $\langle 2\rangle \oplus \langle 6\rangle$

\medskip
--- $A=-3\times \square$, $B=\square$: the points $P+2P^{\sigma}$, $Q$ form a basis of the full Mordell-Weil lattice by the previous argument.

\medskip
\noindent
\textbf{Minimal height basis:} $Q,P+2P^{\sigma}$

\noindent
\textbf{Lattice type:} $\langle 2\rangle \oplus \langle 6\rangle$

\medskip
--- $A=-3\times\square$, $B=-3\times\square$: the points $P+2P^{\sigma}$ and $Q+2Q^{\sigma}$ form a lattice with the Gram matrix $\diag(6,6)$,
which is both $2$ and $3$ saturated.

\medskip
\noindent
\textbf{Minimal height basis:} $P+2P^{\sigma},Q+2Q^{\sigma}$

\noindent
\textbf{Lattice type:} $\langle 6\rangle \oplus \langle 6\rangle$
\subsubsection{When $r_\E=3$}

\subsubsection*{NNYY}
\condbox{0.8\textwidth}{$\Fext{A}{2}\and\Fext{B}{2}\and\IsCube{4AB}{=}\and\IsCube{A}{=}$}
\ratbasis{0}{1}{1}{1}

\medskip
---  $A=\square$, $B=\square$: the Gram matrix for the triple $Q$, $R-S$, $R+S$ is $\diag(2,4,4)$ and the points form a finite index subgroup in $E_{A,B}(\mathbb{Q}(t))$.
The points $Q,R,S$ have the Gram matrix $2I_{3}$ for the identity matrix $I_{3}$ of dimension $3$. Hence the lattice they span is $2$-saturated.

\medskip
\noindent
\textbf{Minimal height basis:} $Q, R, S$.

\noindent
\textbf{Lattice type:} $\langle 2\rangle\oplus\langle 2\rangle\oplus\langle 2\rangle$

\medskip
--- $A=-3\times\square$, $B=\square$: the triple $Q$, $R+S+2(R+S)^{\sigma}$, $R-S$ forms a finite index subgroup in $E_{A,B}(\mathbb{Q}(t))$. The Gram matrix of the triple equals $\diag(2,12,4)$.
We consider a $2$-saturation of this lattice. The new lattice has basis $e_1=Q$, $e_2=R+(R+S)^{\sigma}$, $e_3=-(R-S)$ with the Gram matrix
$$\left(\begin{array}{ccc}
	2  & 0 &0\\
	0  & 4 & -2\\
	0 & -2 &  4\\
\end{array}\right)$$
with determinant $24$. Any non-zero point of the form $\sum_{i=1}^{3}\alpha_i e_i$ for $\alpha_i\in\{0,1\}$ has a height in the set $\{2, 4, 6, 12, 14\}$, hence no non-zero point of this form is further $2$-divisible in this lattice. 

\medskip
\noindent
\textbf{Minimal height basis:} $Q, R+(R+S)^{\sigma}, -(R-S)$.

\noindent
\textbf{Lattice type:} $\langle 2\rangle\oplus A_2(2)$.

\medskip
--- $A=\square$, $B=-3\times\square$:
points $Q+2Q^{\sigma}, R+S,R-S+2(R-S)^{\sigma}$ span a finite index sublattice in $E_{A,B}(\mathbb{Q}(t))$ with the Gram matrix $\diag(6,4,12)$.
We have a relation
$$-3V_{0,0,1}=Q+2Q^{\sigma}+R-S+2(R-S)^{\sigma}$$

\noindent
and the points $Q+2Q^{\sigma}, R+S, V_{0,0,1}$ span a lattice of discriminant $32$. This lattice is generated by $e_1=V_{0,0,1},e_2=Q+2Q^{\sigma}+V_{0,0,1}$ and $e_3=R+S$
with the Gram matrix $\diag(2,4,4)$.
Finally, we check that $e_2+e_3=2W_{0,2,0,1}$ and $Q+2Q^{\sigma}-W_{0,2,0,1}+V_{0,0,1}=W_{0,2,1,1}$ and the points $W_{0,2,1,1},W_{0,2,0,1}$ and $V_{0,0,1}$ form a lattice with the Gram matrix $2 I_{3}$ where $I_3$ is the 3-dimensional identity matrix. Further $2$-divisibility is not possible.

\medskip
\noindent
\textbf{Minimal height basis:} $W_{0,2,1,1},W_{0,2,0,1}, V_{0,0,1}$

\noindent
\textbf{Lattice type:} $\langle 2\rangle\oplus\langle 2\rangle\oplus\langle 2\rangle$

\medskip
--- $A=-3\times\square$, $B=-3\times\square$:
the points $Q+2Q^{\sigma}, R+S+2(R+S)^{\sigma},R-S+2(R-S)^{\sigma}$ span a finite index sublattice with the Gram matrix $\diag(6,12,12)$. 
We find the following linear relations:
$$Q+2Q^{\sigma}+(R-S+2(R-S)^{\sigma})= -3V_{0,0,1}$$
and
$$Q+2Q^{\sigma}+R+S+2(R+S)^{\sigma}+V_{0,0,1} = 2(W_{0,2,0,1}+W_{2,1,0,1}+W_{2,1,1,0}).$$

We check that the Gram matrix of the basis $e_1=Q+2Q^{\sigma}, e_2=W_{0,2,0,1}+W_{2,1,0,1}+W_{2,1,1,0}, e_3=V_{0,0,1}$ equals

$$\left(\begin{array}{ccc}
	6  & 2 &-2\\
	2  & 4 & 0\\
	-2 & 0 &  2\\
\end{array}\right).$$

The matrix defined above has determinant $24$ and we check that non-zero points of the form $\sum_{i=1}^{3}\alpha_i e_i$ such that $\alpha_{i}\in\{0,1\}$ have heights in the set $\{ 2, 4, 6, 12, 14\}$. There are no points of height $3$ in $E_{A,B}(\overline{\mathbb{Q}}(t))$, hence it follows that the basis $\{e_1,e_2,e_3\}$ is $2$-saturated and thus spans the group $E_{A,B}(\mathbb{Q}(t))$.

Let $e_1'=e_3$,$e_2'=e_1+e_3$, $e_3'=-e_2$. This basis has the Gram matrix
$$\left(\begin{array}{ccc}
	2  & 0 &0\\
	0 & 4 & -2\\
	0 & -2 &  4\\
\end{array}\right).$$

\medskip
\noindent 
\textbf{Minimal heights basis:} $V_{0,0,1}, Q+2Q^{\sigma}+V_{0,0,1}, -(W_{0,2,0,1}+W_{2,1,0,1}+W_{2,1,1,0})$

\noindent
\textbf{Lattice type:} $\langle 2\rangle\oplus A_2(2)$.
\subsubsection*{NNYNY}
\condbox{0.8\textwidth}{$\Fext{A}{2}\and\Fext{B}{2}\and\IsCube{4AB}{=}\and\IsCube{A}{\neq}\and\IsCube{B}{=}$}
\ratbasis{1}{0}{1}{1}

--- $A=\square$, $B=\square$:
finite index subgroup spanned by: $P$, $R+S, R-S$

\medskip
\noindent
\textbf{Minimal height basis:} $P, R, S$.

\noindent
\textbf{Lattice type:} $\langle 2\rangle\oplus\langle 2\rangle\oplus\langle 2\rangle$

\medskip

--- $A=\square$, $B=-3\times \square$:
finite index subgroup spanned by: $P$, $R+S$, $R-S+2(R-S)^{\sigma}$

\medskip
\noindent
\textbf{Minimal height basis:} $P,R+S,-(R+(R-S)^{\sigma})$

\noindent
\textbf{Lattice type:} $\langle 2\rangle\oplus A_2(2)$.

\medskip
--- $A=-3\times\square$, $B=\square$:
finite index subgroup spanned by: $P+2P^{\sigma}$, $R+S+2(R+S)^{\sigma}$, $R-S$ with diagonal the Gram matrix $\diag(6,12,4)$.
After $2$ and $3$-saturation we obtain a basis  $P+2P^{\sigma}, X_{1,2,0,1}, R-S$ with the Gram matrix
$$\left(\begin{array}{ccc}
	6 &  2 &  0\\
	2 &  2 &-2\\
	0 & -2 & 4	
\end{array}\right)$$

Equivalent basis (points of height $2$ only): $e_1=P+2P^{\sigma}-2X_{1,2,0,1}-(R-S)$, $e_2=X_{1,2,0,1}$, $e_3=X_{1,2,0,1}+R-S$ with the Gram matrix $\diag(2,2,2)$.
In fact $e_1=U_{0,0,2}$ and $e_3=X_{1, 2, 1, 0}$.

\medskip
\noindent
\textbf{Minimal height basis:} $U_{0,0,2}$, $X_{1,2,0,1}$, $X_{1, 2, 1, 0}$

\noindent
\textbf{Lattice type:} $\langle 2\rangle\oplus\langle 2\rangle\oplus\langle 2\rangle$

\medskip
--- $A=-3\times\square$, $B=-3\times \square$:
finite index subgroup spanned by: $P+2P^{\sigma}, R+S+2(R+S)^{\sigma}, R-S+2(R-S)^{\sigma}$ which has the Gram matrix $\diag(6,12,12)$.  Points $P+2P^{\sigma}$, $R+2R^{\sigma}$ and $S+2S^{\sigma}$ form an overlattice of the previous one with the Gram matrix $\diag(6,6,6)$. Next we observe that
\[-3U_{0,0,1}=P+2P^{\sigma}+R+2R^{\sigma}+S+2S^{\sigma}.\]
So, we find the following basis for the Mordell-Weil subgroup over $\mathbb{Q}(t)$:
\[U_{0,0,1},-U_{0,0,1}-(S+2S^{\sigma}),2U_{0,0,1}+(R+2R^{\sigma})+(S+2S^{\sigma})\]
with the Gram matrix 
$$\left(\begin{array}{ccc}
	2  & 0 &0\\
	0 & 4 & -2\\
	0 & -2 &  4\\
\end{array}\right).$$

\medskip
\noindent
\textbf{Minimal height basis:} $U_{0,0,1},-U_{0,0,1}-(S+2S^{\sigma}),2U_{0,0,1}+(R+2R^{\sigma})+(S+2S^{\sigma})$

\noindent
\textbf{Lattice type:} $\langle 2\rangle\oplus A_2(2)$.

\section{Density on $\mathcal{E}_{A,B}$ with generic rank 0}\label{sec:density_gen_rank_0}

As we have seen in Section \ref{section:previousapproaches}, there are many elliptic surfaces $\mathcal{E}_{A,B}$ with the generic fibre of the form $E_{t}:y^2=x^3+At^6+B$ such that the generic rank is $r_\E=0$:
for those there exists $a,b,c\in\Z$ such that $a,b$ are coprime and such that $A=3a^2c$, $B=cb^2$ and we do NOT have one of the following:
\begin{itemize}
	\item  if $c$ or $-3c$ (resp. $3c$ or $-c$) is a square and $12(abc)^2$ and $3a^2c$ (resp. $12(abc)^2$ and $b^2c$) are cubes (else the generic rank $r_{\E}$ is $2$);
	\item  if $c$ or $-3c$ (resp. $3c$ or $-c$) is a square and either $4AB$ or $A$ (resp. $4AB$ or $B$) are cubes(else $r_{\E}$ is $1$).
\end{itemize}

In the cases where we moreover have $W(E_t)=+1$ on every fibre, then the parity conjecture implies that the rank of the fibres are all even - possibly zero. Using \cite[Lemma A.1. and A.2.]{Desjardins2} and the corresponding tables of values, one can easily determine if a particular elliptic surface $\E_{3ca^2,cb^2}$ has this property. Listing all the cases where we have $r_\E=0$ and $W(\E_t)=+1$ is a tedious task, and we decided to omit it in this paper. This section will focus on giving an example of an alternate proof of the Zariski-density of the rational points that does not involve the generic rank nor the root number.

The following example could already be found in \cite{VA}, based on \cite[Sections 5,6,7]{VA2}: the generic rank is $r_\E=0$ and the root number is constant and equal to $+1$:
\begin{example}
	The elliptic surface $\E_{6\cdot27,6}$ given by the equation $E_{t}:y^2=x^3+6(27t^6+1)$ has as well the property that $W(E_t)=+1$ for all $t\in\P^1(\Q)$. It follows from Theorem \ref{thmdescription} that there are no non-zero $\mathbb{Q}(t)$-rational points in $E_{6\cdot27,6}(\mathbb{Q}(t))$. In this case, our theorem is not sufficient to prove the Zariski density of the rational points on $\E_{6\cdot27,6}$.
	
	It is possible to prove it in a totally different way. Note that the construction of the multisection in \cite{BulthuisVanLuijk} fails to work, because of the difficulty of finding a torsion point on the fibres of the surface. However, Rosa Winter and the first author \cite{DW} construct the following multisection. 
	
	On the surface $\E_{6\cdot27,6}$ we find the following algebraic curve:
	\[\begin{split}
		C: x^3 - \frac{131769}{2704}x^2t^2 + \frac{936903}{1352}xt^4 - \frac{1089}{1352}xt -
		\frac{6223513}{2704}t^6\\
		+ \frac{7743}{1352}t^3 + \frac{16215}{2704}=0.\end{split}\]
	This is a singular curve of genus 1, with a double singularity at the point $[x_0,y_0,t_0]=[22,104,1]$. The desingularisation $\tilde{C}$ of the curve $C$ has a non-singular point $[\frac{12793}{2704},-\frac{2327053}{140608},1]$, thus $\tilde{C}$ is an elliptic curve.
	We observe that $C$ is a multisection (passing through each fibre exactly 3 times), and that $\tilde{C}$ has rank 3 (infinitely many rational points), and moreover that $\E_{6\cdot27,6}$ is the blow-up of a del Pezzo surface of degree 1 (no torsion section). So we can apply \cite[Thm 6.4]{SVL} to conclude the Zariski density of the rational points on $\E_{6\cdot27,6}$.
\end{example} 

\subsection*{Acknowledgments}
We thank Marc Hindry for his suggestion of the proof of Proposition \ref{prop:poly_embed} and for his valuable comments on the preliminary version of this paper. We thank the anonymous referee for valuable suggestions which improved the quality of the paper. We also thank Dino Festi, Ronald van Luijk and Rosa Winter for their helpful comments and suggestions which improved the exposition of this paper. We are grateful to the Hausdorff Center Mathematics in Bonn for the excellent working conditions in April 2018 and to the organizers of the workshop ''Arithmetic of Hyperelliptic Curves'' in September 2017 during which we have initiated our project. The second author is grateful to the University of Bristol for providing us with
the access to Magma cluster CREAM.

BN acknowledges the support by Dioscuri program initiated by the Max Planck Society, jointly managed with the National Science Centre (Poland), and mutually funded by the Polish Ministry of Science and Higher Education and the German Federal Ministry of Education and Research.

JD is partially supported by an NSERC discovery grant.

\end{document}